\documentclass[12pt, a4paper]{amsart}
\usepackage{mathtools} 
\usepackage{hyperref} 
\hypersetup{
    colorlinks=true,       
    linkcolor=blue,          
    citecolor=magenta,        
    filecolor=magenta,      
    urlcolor=cyan           
}

\usepackage{xcolor}
\numberwithin{equation}{section}
\allowdisplaybreaks[4]
\usepackage{amsmath,amsthm,amscd,amssymb,amsfonts, amsbsy}
\usepackage{latexsym}

\numberwithin{equation}{section}
\allowdisplaybreaks

\newtheorem{Theorem}{Theorem}[section]

\newtheorem{Corollary}[Theorem]{Corollary}
\newtheorem{Lemma}[Theorem]{Lemma}

\newtheorem{Remark}[Theorem]{Remark}

\newtheorem{Conjecture}[Theorem]{Conjecture}

\def\XXint#1#2#3{{\setbox0=\hbox{$#1{#2#3}{\int}$}
\vcenter{\hbox{$#2#3$}}\kern-.5\wd0}}

\def\bbZ{\mathbb{Z}}
\def\bbR{\mathbb{R}}

\def\supp{\operatorname{supp}}

\def\re{\operatorname{Re}}

\newcommand{\B}{\mathcal{B}}
\renewcommand{\P}{\mathcal{P}}
\newcommand{\F}{\mathcal{F}}

\newcommand{\M}{\overline{M}}

\DeclareMathOperator*{\essinf}{ess\,inf}
\DeclareMathOperator*{\esssup}{ess\,sup}

\begin{document}

\title[Weighted norm inequalities for rough singular integral operators]{Weighted norm inequalities for rough singular integral operators}

\author[Kangwei Li, Carlos P\'erez, Israel P. Rivera-R\'ios and Luz Roncal]{Kangwei Li, Carlos P\'erez, Israel P. Rivera-R\'ios and Luz Roncal}
\address[K. Li and L. Roncal]{BCAM, Basque Center for Applied Mathematics, Bilbao, Spain}
\email{kli@bcamath.org, lroncal@bcamath.org}
\address[C. P\'erez]{Departamento de Matem\'aticas, Universidad del Pa\'is Vasco UPV/EHU,
IKERBASQUE, Basque Foundation for Science, and BCAM,
Basque Center for
Applied Mathematics, Bilbao, Spain.}
\email{carlos.perezmo@ehu.es}
\address[I. P. Rivera-R\'ios]{Departamento de Matem\'aticas, Universidad del Pa\'{\i}s Vasco UPV/EHU and BCAM,}
\email{petnapet@gmail.com}

\thanks{All the authors are supported by
the Basque Government through the BERC
2018-2021 program and by Spanish Ministry of Economy and Competitiveness
MINECO through BCAM Severo Ochoa excellence accreditation SEV-2013-0323. K.L., C.P. and L.R. are supported by the project MTM2017-82160-C2-1-P funded by (AEI/FEDER, UE) and
acronym ``HAQMEC''. K.L. is also supported by Juan de la Cierva - Formaci\'on 2015 FJCI-2015-24547.
 C.P. and I.P.R-R. are also supported by Spanish Ministry of Economy and Competitiveness MINECO through the project MTM2014-53850-P. I.P.R-R. is also supported by Spanish Ministry of Economy and Competitiveness MINECO through the project MTM2012-30748. L.R. is also supported by Spanish Ministry of Economy and Competitiveness MINECO through the project MTM2015-65888-C04-4-P and by 2017 Leonardo grant for Researchers and Cultural Creators, BBVA Foundation. The Foundation accepts no responsibility for the opinions, statements and contents included in the project and/or the results thereof, which are entirely the responsibility of the authors.}

\date{\today}


\begin{abstract} In this paper we provide weighted estimates for rough operators, including rough homogeneous singular integrals $T_\Omega$ with $\Omega\in L^\infty(\mathbb{S}^{n-1})$ and the Bochner--Riesz multiplier at the critical index $B_{(n-1)/2}$. More precisely, we prove qualitative and quantitative versions of Coifman--Fefferman type inequalities and their vector-valued extensions,  weighted $A_p-A_\infty$ strong and weak type inequalities for $1<p<\infty$,  and $A_1-A_\infty$ type weak $(1,1)$ estimates. Moreover, Fefferman--Stein type inequalities are obtained, proving in this way a conjecture raised by the second-named author in the 90's. As a corollary, we obtain the weighted  $A_1-A_\infty$ type estimates.
Finally, we study rough homogenous singular integrals with a kernel involving a function $\Omega\in L^q(\mathbb{S}^{n-1})$, $1<q<\infty$,
and provide Fefferman--Stein inequalities too. The arguments used for our proofs combine several tools: a recent sparse domination result by Conde--Alonso et al. \cite{CACDPO}, results by the first author in \cite{L}, suitable adaptations of Rubio de Francia algorithm, the extrapolation theorems for $A_{\infty}$
weights \cite{CMP,CGMP} and ideas contained in previous works by A. Seeger in \cite{S} and D. Fan and S. Sato  \cite{FS}.

\end{abstract}

\maketitle

\section{Introduction and main results}
\label{sec:intro}

Many important inequalities in Harmonic
Analysis and P.D.E. are of the form
\begin{equation*}
\int_{\mathbb{R}^{n}} |Tf(x)|^p\, w(x)\,d x
\le
C\,\int_{\mathbb{R}^{n}} |Sf(x)|^p\, w(x)\,d x,
\end{equation*}
where typically $T$ is an operator that carries some degree of singularity
(e.g., some singular integral operator) and $S$ is an operator
which is, desirably, easier to handle (e.g., a maximal operator), and $w$ is
in some class of weights. One of the most usual techniques
for proving such results is to establish a good-$\lambda$
inequality between $T$ and $S$. This method, due to D. L.
Burkholder and R. F. Gundy \cite{BG}, relies on the comparison of the
measure of the level sets of $S$ and $T$, namely on finding an intrinsic constant $c>0$ (depending upon the operators $S,T$) such  that  for every $\lambda>0$ and small $\varepsilon>0$,
\begin{equation} \label{intro2}
w \big\{y\in \mathbb{R}^{n}: |T f(y)| > 3\,\lambda, |S f(y)| \le \varepsilon\,\lambda\, \big\}
\le
c\, \varepsilon\, w \big\{ y\in \mathbb{R}^{n}: |T f(y)| > \lambda \big\},
\end{equation}
where the weight $w$ is usually assumed to be in the Muckenhoupt class $A_\infty$.

A paradigmatic example of the application of that technique was provided by R. R. Coifman and C. Fefferman in the classical paper \cite{CF}. In that work they proved that, given $w\in A_\infty$ and any $p$, $0<p<\infty$, there is a constant $c$ depending on $p$ and
$w$ such that
\begin{equation}\label{coifman-fefferman}
\int_{\mathbb{R}^{n}} |Tf(x)|^p\, w(x)\,d x
\leq
c\,\int_{\mathbb{R}^{n}} Mf(x)^p\, w(x)\,d x,
\end{equation}
for any function $f$ such that the left-hand side is finite. Here, $T$ is any Calder\'on-Zygmund operator and $M$ is the Hardy--Littlewood maximal operator. We point out that this estimate does not hold in general for every singular integral: there are examples of convolution type operators with kernels satisfying the H\"ormander smoothness condition for which \eqref{coifman-fefferman} fails, as it can be found in \cite{MPT}. In particular, it is impossible to establish  a good-$\lambda$ inequality between these operators and $M$. As an immediate consequence of the classical Muckenhoupt's theorem,
if $T$ satisfies \eqref{coifman-fefferman} then $T$ is bounded on $L^p(w)$, $1<p<\infty$, $w\in A_p$ (for the  notations  and basic facts on $A_p$ weights, see Subsection \ref{AptheoryOfWeights}).

On the other hand, it is not so well known that the estimate \eqref{coifman-fefferman} turns out to be the key estimate for proving the following non-standard two-weight result for $T$,
namely
\begin{equation} \label{Iterated-p}
\|T f\|_{L^p(w)}\leq c_{p,T} \|f\|_{L^p(M^{\lfloor p\rfloor+1} w)},
\end{equation}
as shown in \cite{P0}. The notation $\lfloor p\rfloor$ means the integer part of $p$, and by $M^{\lfloor p\rfloor+1}$ we mean the $(\lfloor p\rfloor+1)$-fold composition of the operator $M$. Actually, the method in \cite{P0} is very general. Roughly, if $p\in (1,\infty)$ and if $T$ is a linear operator whose adjoint $T^t$  satisfies \eqref{coifman-fefferman} with exponent $p'$ and for any $RH_{\infty}$ weight, then \eqref{Iterated-p} holds. The proof can be reduced, after using the duality in $L^p$, to study a corresponding estimate for the maximal Hardy-Littlewood function proved in \cite{P}.  Estimates like \eqref{Iterated-p}, using the iteration of the maximal functions to control singular integrals were considered first by J. M. Wilson in \cite{W2}. It is proved in his work that \eqref{Iterated-p} holds  for $p \in (1,2)$ when $T$ is any  smooth singular integrals of convolution type. Wilson's method is different and it is based on proving corresponding square function estimates.

Even more, the estimate \eqref{coifman-fefferman} is crucial as well in the solution of Sawyer's conjecture in \cite{CMP-IMRN}. This time, the fact that the estimate holds for any $0 <p<1$ and for any $w\in A_{\infty}$ plays a fundamental role.

In this paper we shall start by considering Coifman--Fefferman's type estimates like \eqref{coifman-fefferman}, and then we will show other qualitative and quantitative weighted estimates, in the case where $T$ is either a rough homogeneous singular integral or the Bochner--Riesz multiplier at the critical index. By ``qualitative'' we mean weighted inequalities without specifying the dependence of the norm bound on the $A_p$ constant, while in ``quantitative'' estimates we search for the optimal explicit dependence on such constant.

We recall that
given $\Omega \in L^1(\mathbb{S}^{n-1})$ such that $\int_{\mathbb{S}^{n-1}}\Omega=0$,
we can define a kernel
\[
K(x)=\frac{\Omega(x')}{|x|^{n}}
\]
where $x'=\frac{x}{|x|}$. It is clear that $K$ is homogeneous of
degree $-n$. Using that kernel we define the rough homogeneous singular integral $T_\Omega$ by
\begin{equation}\label{eq:TOmega}
T_{\Omega}f(x)=\text{p.v.}\int_{\mathbb{R}^{n}}\frac{\Omega(x')}{|x|^{n}}f(x-y)dy.
\end{equation}
On the other hand, the Bochner--Riesz multiplier at the critical index $B_{(n-1)/2}$ is defined by
\begin{equation}
\label{eq:BR}
\widehat{B_{(n-1)/2}(f)}(\xi)= (1-|\xi|^2)_+^{(n-1)/2}\hat f(\xi).
\end{equation}

It is well known that $T_{\Omega}$, with $\Omega \in L^{\infty}(\mathbb{S}^{n-1})$, is bounded on $L^p(w)$, $1<p<\infty$,  when $w\in A_p$ (first proved in the celebrated paper by J. Duoandikoetxea and J. L. Rubio de Francia \cite{DR}, later improved in \cite{DuoTAMS} and \cite{Wa}, and with quantitative version in \cite{HRT}). In view of these works, the second  author conjectured after \cite{P0} that \eqref{Iterated-p} would hold for rough singular integral operators $T_{\Omega}$ in the case
$\Omega \in L^{\infty}(\mathbb{S}^{n-1})$ (the conjecture is made explicit in \cite{PRRR}).  Indeed, the method mentioned above could not be used since \eqref{coifman-fefferman} was not available. In particular, no good-$\lambda$ estimate relating $T_{\Omega}$ and $M$  like \eqref{intro2} is known to hold (See Conjecture \ref{Conj1} below). Similar comments can be made for $B_{(n-1)/2}$, which is also a bounded operator on $L^p(w)$, with $w\in A_p$ (see \cite{SS}).

\subsection{Qualitative estimates}

In this subsection we present some new weighted estimates for rough operators and we also present some known results for such operators, as consequences of the former.
The first result to be shown will be (despite the title of the present subsection), a quantitative version of the Coifman--Fefferman's inequality \eqref{coifman-fefferman}, for $1\le p<\infty$.

\begin{Theorem} \label{ThmTfMf}
Let $T$ be either $T_\Omega$ with $\Omega \in L^\infty$ satisfying $\int_{\mathbb{S}^{n-1}}\Omega =0$ or $B_{(n-1)/2}$.  Let $p\in [1,\infty)$ and let $w\in A_\infty$, then
\begin{equation}\label{T/Mp=1}
\|Tf\|_{L^{p}(w)}
\leq
c_{p,T}   [w]_{A_{\infty}}^2\, \|M f\|_{L^{p}(w)}
\end{equation}
for any smooth function such that the left-hand side is finite.
\end{Theorem}

The novelty here is that we avoid completely the use of the good-$\lambda$ method. Indeed, we combine the sparse formula in Theorem \ref{Thm:Sparse} below from \cite{CACDPO}, together with a Carleson embedding type argument in the case $p=1$ and
the technique of principal cubes introduced in \cite{MW} for the case $p>1$.

A natural question is wether estimate \eqref{T/Mp=1} holds as well for $0<p<1$.  Indeed, this is true in this range and it follows from the case $p=1$ by means of an extrapolation theorem for $A_{\infty}$ weights from \cite{CMP, CGMP} as stated in the next Corollary. The difference is that in this case the results are just qualitative since it is not clear how to obtain good bounds from the extrapolation method. On the other hand, the method is very flexible allowing many other spaces and further extensions.

\begin{Corollary} \label{extrapcorollary}
Let $T$ be either $T_\Omega$ with $\Omega \in L^\infty$ satisfying $\int_{\mathbb{S}^{n-1}}\Omega =0$ or $B_{(n-1)/2}$. Let $p,q\in (0,\infty)$ and $w\in A_{\infty}$. There is a constant $c$ depending on the $A_{\infty}$ constant such that:
\begin{itemize}
\item[a) ] {\bf Scalar context}.
\begin{equation}\label{T/Mp}
\|Tf\|_{L^{p}(w)}
\leq
c\, \|M f\|_{L^{p}(w)}
\end{equation}
and
\begin{equation}\label{weaknormp-p}
\|Tf\|_{L^{p,\infty}(w)}
\leq
c\, \|M f\|_{L^{p,\infty}(w)},
\end{equation}
for any smooth function such that the left-hand side is finite.

\item[b) ] {\bf Vector-valued extension}.
\begin{equation} \label{intro4}
\Big\|
\Big(\sum_j |T f_j|^q\Big)^{1/q}
\Big\|_{L^p(w)}
\le
c\, \Big\| \Big(\sum_j (Mf_j)^q\Big)^{1/q} \Big\|_{L^p(w)}
\end{equation}
and
\begin{equation} \label{intro6}
\Big\| \Big(\sum_j |T f_j|^q\Big)^{1/q} \Big\|_{L^{p,\infty}(w)}
\le
c\, \Big\| \Big(\sum_j (Mf_j)^q\Big)^{1/q}
\Big\|_{L^{p,\infty}(w)},
\end{equation}
for any smooth vector function such that the left-hand side is finite.
\end{itemize}

\end{Corollary}

\begin{Conjecture}\label{Conj1}
We conjecture that the constant in \eqref{T/Mp=1} (or \eqref{T/Mp} and \eqref{weaknormp-p}) is a multiple of $p\,[w]_{A_{\infty}}$, as in the case of Calder\'on-Zygmund operators.
In the latter case, the proof of such an estimate relies upon a variation, obtained in  \cite[p. 152]{LOP2}, of the following sharp exponential estimate from Buckley's work \cite{Bu}: for every $\lambda>0$ and small $\varepsilon>0$
$$
\big| \big\{y\in \mathbb{R}^{n}: |T f(y)| > 3\,\lambda, Mf(y) \le \varepsilon\,\lambda\, \big\}\big|
\le
c\, e^{-c/\varepsilon} \,  \big|\{ y\in \mathbb{R}^{n}: Mf(y) > \lambda \big\}\big|.
$$
We also conjecture that this good-$\lambda$ estimate with exponential decay holds for the operators in Theorem \ref{ThmTfMf} and hence the linear dependence on the $A_{\infty}$ constant holds as well.
\end{Conjecture}

Similar estimates as in the corollary hold for many other spaces $X(w)$  where $X$ is an appropriate  generalized rearrangement-invariant function space like Orlicz spaces, classical or generalized Lorentz spaces or Marcinkiewicz spaces, instead of strong or weak or $L^p$ spaces. All these are consequences of the main results from \cite[Theorem 2.1]{CGMP}, where modular type estimates can be found as well, see \cite[Theorem 3.1]{CGMP}.  A similar result holds for (unweighted) variable $L^p$ spaces following the main idea from \cite{CFMP}. See Appendix \ref{A} for some special examples.

Of course, Corollary \ref{extrapcorollary} improves in several directions the main result in \cite{DR}, namely if $w\in A_p$, $p>1,$ then $T: L^{p}(w) \to L^{p}(w)$ continuously. We also remark that it is not clear how to prove estimates \eqref {intro4} or \eqref {intro6} directly from the good-$\lambda$ between $T$ and $M$ even in the classical situation of $T$ being a Calder\'on-Zygmund operator $T$.

As for the case $p=1$, from \eqref{weaknormp-p} we have the following:
\begin{equation*}
\|Tf\|_{L^{1,\infty}(w)}
\leq
c_w\, \|M f\|_{L^{1,\infty}(w)}, \qquad w \in A_{\infty},
\end{equation*}
with constant $c_w$ depending on the $A_{\infty}$ constant of $w$. Then by the classical Fefferman-Stein's inequality, we obtain
\begin{equation*}
\|Tf\|_{L^{1,\infty}(w)}
\leq
c_{w}\, \|f\|_{L^{1}(Mw)},
\end{equation*}
and hence, if further $w\in A_1$, we have
\begin{equation} \label{roughA1}
\|T\|_{L^1(w)\rightarrow L^{1,\infty}(w)}
\leq
c_{w}\, [w]_{A_1},
\end{equation}
namely, if $w\in A_1$, then $ T: L^{1}(w) \to L^{1,\infty}(w)$ \,
with the operator norm bound depending on the $A_1$ constant of the weight $w$. Inequality \eqref{roughA1} yields a different and improved proof of \cite[Thm. E, Appendix B]{CACDPO} and also a new qualitative proof of the weighted weak type $(1,1)$ estimate for $T_\Omega$ established in \cite{FS}. Furthermore, a similar, and new result holds for the vector-valued extension of \eqref{roughA1} combining \eqref{intro6} with the vector-valued extension of the Fefferman--Stein's inequality for the maximal function established in \cite{P-TAMS}. More precisely we have for $q\in(1,\infty)$ and $w \in A_{\infty}$
\begin{equation*}
\Big\| \Big(\sum_j |T f_j|^q\Big)^{1/q} \Big\|_{L^{1,\infty}(w)}
\le
c\, \Big\| \Big(\sum_j (f_j)^q\Big)^{1/q}
\Big\|_{L^{1}(Mw)}.
\end{equation*}

However, and by means of other methods, a much more precise quantitative version of \eqref{roughA1} will be provided in Subsection \ref{Sec:quant} (see Theorem \ref{Thm:weak11}).

We finish this subsection with two more non-standard consequences from  \eqref{T/Mp}. The first one is the following two-weight estimate which solves affirmatively the conjecture mentioned above about inequality  \eqref{Iterated-p} for rough singular integral operators.

\begin{Corollary}\label{qualitativeiterationp-p}
Let $T$ be either $T_\Omega$ with $\Omega \in L^\infty$ satisfying $\int_{\mathbb{S}^{n-1}}\Omega =0$ or $B_{(n-1)/2}$. If $p\in (1,\infty)$, then
\begin{equation*}
 \label{Iterated-p-rough}
\|T f\|_{L^p(w)}\leq c_{n,p,T} \|f\|_{L^p(M^{\lfloor p\rfloor+1} w)}  \qquad w\geq 0.
\end{equation*}
The above inequality is sharp in the sense that we cannot replace $M^{\lfloor p\rfloor+1}$ by $M^{\lfloor p\rfloor}$.
\end{Corollary}

We remark that a very interesting  similar result to Corollary \ref{qualitativeiterationp-p} was obtained recently by D. Beltran in \cite[Corollary 1.4]{B} for the Carleson operator
$\mathcal{C}$. A bit surprisingly, the Carleson operator cannot satisfy neither an inequality like \eqref{coifman-fefferman} nor a good-$\lambda$ inequality between $\mathcal{C}$  and the Hardy-Littlewood maximal function $M$: otherwise, $\mathcal{C}$ would be of weak type $(1,1)$ (since estimate  \eqref{weaknormp-p} would hold) but, as  it is well known, this property is false.

A quantitative version of Corollary \ref{qualitativeiterationp-p}, whose proof requires a different argument, will be provided in Corollary \ref{Cor:Conj} of Subsection \ref{Sec:quant}.

The second non-standard consequence from  \eqref{T/Mp} is that we can extend  the conjecture formulated by E. Sawyer  \cite{Sa} for the Hilbert transform to rough singular integrals. E. Sawyer proved for the maximal function in the real line that if $u,v\in A_1$ then
\begin{equation} \label{SawyerProblem}
 \Big\|\frac{Mf}{v }\Big\|_{L^{1,\infty}(uv)}\leq c\|f\|_{L^1(uv)}
\end{equation}
and posed the question whether a similar estimate with $M$ replaced by the Hilbert transform would hold or not. A positive answer to this question was given in \cite{CMP-IMRN} where a more general version of this problem was obtained for Calder\'on-Zygmund operators and the maximal function in higher dimensions. Furthermore, the main result of \cite{CMP-IMRN} also solved and extended conjectures proposed by Muckenhoupt-Wheeden in \cite{MW} enlarging the class of weights for which this estimate holds, namely $u\in A_1$, and $v\in A_1$ or $uv\in A_\infty $. This was further generalized in \cite{OP}.

Very recently, a conjecture extending the one proposed by Sawyer and raised in \cite{CMP-IMRN}, has been solved by the first two authors together with  S. Ombrosi. This new recent result extends the class of weights for which Sawyer's inequality \eqref{SawyerProblem} holds and it is the following (see \cite{LiOP}).

\begin{quote}\label{thm:main} Let $u\in A_1$ and $v\in A_\infty$.  Then there is a finite constant $c$ depending on the $A_1$ constant of $u$ and the $A_{\infty}$ constant $v$ such that
\begin{equation} \label{BigConjecture}
\Big\|\frac{ M(fv)} {v}\Big\|_{L^{1,\infty}(uv)}\leq c \|f\|_{L^1(uv)}.
\end{equation}
\end{quote}

Using this result we have the following.

\begin{Theorem} Let $T$ be either $T_\Omega$ with $\Omega \in L^\infty$ satisfying $\int_{\mathbb{S}^{n-1}}\Omega =0$ or $B_{(n-1)/2}$. Let $u\in A_1$ and suppose that $v$ is a
weight such that for some $\delta>0$, $v^{\delta} \in A_{\infty}$. Then, there is a constant $c$ such that
\begin{equation} \label{extrapolweak}
 \Big\| \frac{Tf}{v }\Big\|_{L^{1,\infty}(uv)}
 \leq
 c\,\Big\| \frac{Mf}{v }\Big\|_{L^{1,\infty}(uv)}.
\end{equation}
Hence, if $u \in A_1$ and $v \in A_{\infty}$, then
there is a constant $c$ such that
\begin{equation}\label{sawyer en mas dimensiones}
\left\|\frac{T(fv)}{v}\right\|_{L^{1, \infty}(uv)} \leq c\;\|f\|_{L^1(uv)}.
\end{equation}
\end{Theorem}

\begin{proof}
The proof of \eqref{extrapolweak} is a corollary of  \eqref{T/Mp} (actually the range $p\in (0,1)$ is the relevant one) after applying  \cite[Thm. 1.1]{CMP} or the more general case \cite[Thm. 2.1]{CGMP}.

On the other hand, combining \eqref{BigConjecture} together with \eqref{weaknormp-p} (which we recall that it follows from \eqref{T/Mp=1}), the inequality \eqref{sawyer en mas dimensiones} holds.

\end{proof}

\subsection{Quantitative estimates}\label{Sec:quant}
In the last decade, plenty of works about weighted estimates
have been devoted to the study of the quantitative dependence on the
$A_{p}$ constant, on the $A_{1}$ constant, and also on mixed constants
involving the $A_{\infty}$ constant, of the
weighted $L^{p}$ boundedness constant of several operators. Quite recently, some results
in that direction for rough singular integrals have appeared in works
such as \cite{CACDPO,HRT,PRRR}. Motivated by the latter (in particular by the most recent \cite{CACDPO}), in this section we present several results showing improvements on the dependence on the $A_p$ and $A_{1}$ constant of $T$, where $T$ is either a rough homogeneous singular integral or the Bochner--Riesz multiplier at the critical index. Our first result regards the improved $A_p$ type estimate:

\begin{Theorem}\label{Thm:improvedAp}
Let $T$ be either $T_\Omega$ with $\Omega \in L^\infty$ satisfying $\int_{\mathbb{S}^{n-1}}\Omega =0$ or $B_{(n-1)/2}$, and $w\in A_p$. Let us denote $\sigma:=w^{\frac{1}{1-p}}$. Then
\[
\|T\|_{L^p(w)}\le C_T [w]_{A_p}^{\frac{1}{p}}([w]_{A_\infty}^{\frac 1{p'}}+ [\sigma]_{A_\infty}^{\frac 1p})\min\{[w]_{A_\infty}, [\sigma]_{A_\infty}\},\quad 1<p<\infty.
\]
In particular,
\[
\|T\|_{L^p(w)}\le C_T [w]_{A_p}^{\frac{p}{p-1}},\quad 1<p<\infty.
\]
We also have the following weak type estimate
\[
\|T\|_{L^p(w)\rightarrow L^{p,\infty}(w)}\le C_T [w]_{A_p}^{\min \{2, \frac p{p-1}\}}.
\]
\end{Theorem}
\begin{Remark}
Since
\[
\max\Big\{1, \frac 1{p-1}\Big\}< \frac p{p-1}\le \frac 1{p-1}\max\{2,p\}\le 2\max\Big\{1, \frac 1{p-1}\Big\},
\]
our bound improves the known result in \cite{HRT} and also the very recent result in \cite{CACDPO}, however we don't reach the exponent $\max\{1, \frac 1{p-1}\}$ provided it is possible to obtain such an estimate.
\end{Remark}

For the following theorem, we refer Subsection \ref{subsec:Young} for definitions and details related to Young functions and associated maximal functions.
\begin{Theorem}
\label{Thm:A1}
Let $1<p<\infty$ and $A$ be a Young function. Let $T$ be either $T_\Omega$ with $\Omega \in L^\infty$ satisfying $\int_{\mathbb{S}^{n-1}}\Omega =0$ or $B_{(n-1)/2}$. Then, for any $f\in C_c^{\infty}(\mathbb{R}^n)$,
\begin{equation}
\|Tf\|_{L^{p}(w)}\leq C_T(p')^{2}\|M_{\bar{A}}\|_{L^{p'}}\|f\|_{L^{p}(M_{A_p}(w))}.\label{eq:MA}
\end{equation}
\end{Theorem}
From the preceding theorem, by using \eqref{eq:MaBdd} in below, if we choose $A(t)=t^p\left(1+\log^+t\right)^{p-1+\delta}$  with $\delta\in(0,1]$ we obtain the following result, which is in turn a quantitative version of Corollary \ref{qualitativeiterationp-p}.
\begin{Corollary} \label{Cor:Conj} Let $1<p<\infty$ and $T$ be either $T_\Omega$ with $\Omega \in L^\infty$ satisfying $\int_{\mathbb{S}^{n-1}}\Omega =0$ or $B_{(n-1)/2}$. Then, for any $f\in C_c^{\infty}(\mathbb{R}^n)$,
\begin{equation}
\|Tf\|_{L^{p}(w)}\leq C_T(p')^{2}p^2\big(\frac{1}{\delta}\big)^\frac{1}{p'}\|f\|_{L^{p}(M_{L(\log L)^{p-1+\delta}}w)}.\label{eq:LlogL}
\end{equation}
The inequality above is sharp in the sense that $\delta =0$ is false.
\end{Corollary}
At this point we conjecture that \eqref{eq:LlogL} should hold with $pp'$ instead of $(pp')^2$.
We can also derive an improvement of some results obtained in \cite{PRRR} concerning the $A_1$ constant.

\begin{Corollary} \label{Cor:A1} Let $1<p<\infty$ and $T$ be either $T_\Omega$ with $\Omega \in L^\infty$ satisfying $\int_{\mathbb{S}^{n-1}}\Omega =0$ or $B_{(n-1)/2}$. Then, for any $f\in C_c^{\infty}(\mathbb{R}^n)$,
\begin{equation}\label{eq:Mr}
\|Tf\|_{L^{p}(w)}\leq C_Tp(p')^{2}(r')^\frac{1}{p'}\|f\|_{L^{p}\left(M_r(w)\right)}.
\end{equation}
If, moreover, $w\in A_{\infty}$ then
\begin{equation}\label{eq:M}
\|Tf\|_{L^{p}(w)}\leq C_Tp(p')^{2}[w]_{A_{\infty}}^{\frac{1}{p'}}\|f\|_{L^{p}(Mw)}.
\end{equation}
Furthermore, if $w\in A_{1}$ then
\begin{equation}
\label{eq:A1}
\|T\|_{L^{p}(w)}\leq C_Tp(p')^{2}[w]_{A_{1}}^{\frac{1}{p}}[w]_{A_{\infty}}^{\frac{1}{p'}} \leq C_Tp(p')^{2}[w]_{A_{1}}.
\end{equation}

Also, as a direct consequence of \cite[Corollary 4.3]{D}, if $w\in A_{q}$,
for $1\leq q<p$, then
\[
\|T\|_{L^{p}(w)}\leq c_{n,p,q}C_T[w]_{A_{q}}.
\]
In all the inequalities above, $C_T$ is the one in \eqref{eq:CT}.
\end{Corollary}

We would like to point out the fact that similar results for Carleson operators were obtained in \cite{DL} and later on in \cite{B}.

In the previous subsection we showed that a new qualitative proof of the endpoint estimate obtained in \cite{FS} could be obtained via extrapolation (see the explanation just after Corollary \ref{extrapcorollary}). Now we present a quantitative version of the weighted weak type $(1,1)$ estimate, which is formulated as follows.
\begin{Theorem}\label{Thm:weak11}
Let $w\in A_1$ and $T$ be either $T_\Omega$ with $\Omega \in L^\infty$ satisfying $\int_{\mathbb{S}^{n-1}}\Omega =0$ or $B_{(n-1)/2}$. Then
\[
\|T\|_{L^1(w)\rightarrow L^{1,\infty}(w)}\le C_T [w]_{A_1}[w]_{A_\infty}\log_2([w]_{A_\infty}+1).
\]
\end{Theorem}

Let us compare this result with the corresponding result for Calder\'on-Zygmund operators, which was first proved in \cite{LOP2} (see also \cite{LOP,HP,DLR}) and which is optimal, see \cite{LNO}.  In our case, we have the extra constant  $[w]_{A_\infty}$  due to the roughness of the kernel. However, we believe that this is the best possible bound we can expect with the current techniques.
Indeed, it is not explicitly stated in \cite{V, FS, FS1, CD} how the weighted weak type $(1,1)$ bound depends on the constant. But it is possible to check that our result improves the implicit constant obtained in those papers.
Finally, we also study independently $T_{\Omega}$ for the case of $\Omega \in L^{q}(\mathbb{S}^{n-1})$, $1<q<\infty$ (see \cite{DuoTAMS} and \cite{Wa} for more backgrounds). Specifically, we prove the following result for sparse operator $\mathcal A_{r, \mathcal S}$, which can be interesting by itself (for definitions and basics about sparse families, see Subsection \ref{subsec:known}).
\begin{Theorem}\label{Thm:omegalq}
Let $r>1$, $w$ be a weight and $\mathcal S$ be a sparse family. Let $A$ be a Young function such that $\bar A \in B_{p'}$. For $f\ge 0$, set
\[
\mathcal A_{r, \mathcal S}(f)(x)= \sum_{Q\in \mathcal S} \langle f ^r \rangle_Q^{\frac 1r}\chi_Q(x).
\]
Then for $p>r$, there holds
\[
\|\mathcal A_{r,\mathcal S}(f)\|_{L^p(w)}\le (c_n p')^{\frac {r(p-1)}{p-r}}  \big(\frac p{r}\big)'  \|M_{\bar{A}}\|_{L^{p'}}\|f\|_{L^p(M_{A_p}w)}.
\]
\end{Theorem}
We remark that the qualitative version of Theorem \ref{Thm:omegalq} is also obtained by Beltran \cite{B} by using two weight bump theorem. We shall give two proofs for Theorem \ref{Thm:omegalq}, one using the Rubio de Francia algorithm and the other one given in the Appendix~\ref{B} using the two weight bump theorem.  Combining Theorem \ref{Thm:omegalq} and the sparse domination principle in \cite{CACDPO}, we obtain the following:
\begin{Theorem}\label{Thm:Q}
Given $1<q<\infty$, let $\Omega\in L^{q,1}\log L(\mathbb{S}^{n-1})$ have zero average and $w$ be a weight. Let $A$ be a Young function such that $\bar A \in B_{p'}$. Then for $p>q'$, there holds
\[
\|T_\Omega(f)\|_{L^p(w)}\le c_n q \|\Omega\|_{L^{q,1}\log L(\mathbb{S}^{n-1})}  \big(\frac p{q'}\big)' (c_n p')^{\frac {q'(p-1)}{p-q'}} \|M_{\bar{A}}\|_{L^{p'}}\|f\|_{L^p(M_{A_p}w)},
\]
for any $f\in C_c^{\infty}(\mathbb{R}^n)$.
\end{Theorem}
By $\Omega\in L^{q,1}\log L(\mathbb{S}^{n-1})$ we mean that the following norm is finite (cf. \cite{CACDPO})
$$
\|\Omega\|_{L^{q,1}\log L(\mathbb{S}^{n-1})}:=q\int_0^{\infty}t \log(e+t)|\{\theta\in \mathbb{S}^{n-1}:|\Omega(\theta)|>t\}|^{\frac{1}{q}}\frac{dt}{t}.
$$

Then immediately we have the following estimate.
\begin{Corollary}
Given $1<q<\infty$, let $\Omega\in L^{q,1}\log L(\mathbb{S}^{n-1})$ have zero average and $w$ be a weight. Then for $p>q'$, we have
\begin{equation}
\label{eq:Tomr}
\|T_{\Omega} f\|_{L^p(w)}\leq c_{n,p,q} \|\Omega\|_{L^{q,1}\log L(\mathbb{S}^{n-1})}  \|f\|_{L^p(M^{\lfloor p\rfloor+1} w)}.
\end{equation}
\end{Corollary}
Moreover, when $A(t)=t^{pr}$, we obtain the following estimate:
\begin{Corollary}
Given $1<q<\infty$, let $\Omega\in L^{q,1}\log L(\mathbb{S}^{n-1})$ have zero average and $w$ be a weight. If $1<r<\infty$, then for $p>q'$
\[
\|T_\Omega(f)\|_{L^p(w)}\le c_n q \|\Omega\|_{L^{q,1}\log L(\mathbb{S}^{n-1})} p  (r')^{\frac 1{p'}}\Big(\frac p{q'}\Big)' (c_n p')^{\frac {q'(p-1)}{p-q'}}  \|f\|_{L^p(M_{r}w)},
\]
which immediately implies
\[
\|T_\Omega(f)\|_{L^p(w)}\le c_{n,p,q} \|\Omega\|_{L^{q,1}\log L(\mathbb{S}^{n-1})} [w]_{A_1}^{\frac 1p}[w]_{A_\infty}^{\frac 1{p'}}\|f\|_{L^p(w)},\quad p>q'.
\]
Then we also have that
\[
\|T_\Omega(f)\|_{L^p(w)}\le c_{n,p,q} \|\Omega\|_{L^{q,1}\log L(\mathbb{S}^{n-1})} [w]_{A_1}\|f\|_{L^p(w)},\quad p>q'.
\]
and as a direct consequence of \cite[Corollary 4.3]{D}, if $w\in A_{s}$, with $1\leq s<p$ then
\[
\|T_\Omega(f)\|_{L^p(w)}\le c_{n,p,q,s} \|\Omega\|_{L^{q,1}\log L(\mathbb{S}^{n-1})} [w]_{A_s}\|f\|_{L^p(w)},\quad p>q'.
\]
\end{Corollary}

We sketch now the ideas used to prove our results. For the proof of  Theorem \ref{ThmTfMf}, we prove the result for $p=1$ and $p>1$ separately. In both cases, the starting point is the sparse domination by Conde--Alonso e.a. \cite{CACDPO}, whose statement we recall in Subsection \ref{subsec:known} for the sake of completeness. Then, for $p=1$, a Carleson embedding type argument is used, whereas for $p>1$ we use an argument involving principal cubes. The refinement in the strong $A_p$ estimate and the weak type inequality contained in Theorem \ref{Thm:improvedAp} are deduced by taking the sparse domination in \cite{CACDPO} and then we follow arguments in \cite{L}. The proof of Theorem \ref{Thm:A1} is a combination of Rubio de Francia algorithm and again the sparse domination in  \cite{CACDPO}. The weighted weak type $(1,1)$ estimate of Theorem \ref{Thm:weak11} is based on Corollary \ref{Cor:A1} (which follows as a consequence of Theorem  \ref{Thm:A1}), the strategy used by Seeger in \cite{S} and the approach in the work by Fan and Sato \cite{FS} (based, in its turn, on the previous works by Seeger \cite{S} and Vargas \cite{V}). Finally, the result concerning rough homogeneous singular integrals when $\Omega \in L^{q,1}\log L(\mathbb{S}^{n-1})$ merges the Rubio de Francia algorithm, the sparse domination and again ideas in \cite{L}.

The structure of the paper is as follows. In Section \ref{sec:key} we collect several definitions and known results that will be the cornerstones in our proofs. Moreover, some aspects of Young functions and associated maximal functions are expounded. In section \ref{Sec:ProofThmTfMf} we prove Theorem \ref{ThmTfMf}. Theorem \ref{Thm:improvedAp} is proven in Section \ref{sec:improvedAp}. Section \ref{sec:ThmA1} contains the proof of Theorem~\ref{Thm:A1}. The proof of Theorem \ref{Thm:weak11} is shown in Section \ref{sec:Thmweak11} and the result regarding $\Omega\in L^{q,1}\log L(\mathbb{S}^{n-1})$ in Theorem \ref{Thm:omegalq} is included in Section~\ref{sec:Thmomegalq}.

Throughout the paper we will use fairly standard notation. By $c, c_n, c_{T}\ldots$ we mean positive constants that are either universal or depending on the subindices, but not depending on the essential variables. These constants may vary at each occurrence. For an operator $T$, we will denote by $\|T\|_{B_1\to B_2}$, or just $\|T\|_{B_1}$ if $B_1=B_2$, the norm of the operator, i.e., the least constant $N$ such that $\|Tf\|_{B_2}\le N\|f\|_{B_1}$.
We will denote the average of a function $f$ over a cube $Q$ by $\langle f \rangle_Q:=|Q|^{-1}\int_Qf(x)\,dx$. For any function $f$ and a weight $w$, we shall use $\langle f \rangle^w_Q:=w(Q)^{-1}\int_Qf(x)w(x)\,dx$, where $w(Q):=\int_Q w(x)\,dx$. Moreover, for $1<s < \infty$, the notation $\langle f \rangle_{s,Q}$ means $(\langle |f|^s \rangle_{Q})^{1/s}$.

\section{Some definitions and key results}
\label{sec:key}
In this section we gather some results and definitions that will be fundamental for the proofs of our main results.

\subsection{Some basics of the $A_p$ theory of weights}
\label{AptheoryOfWeights}

For $1<p<\infty$, we say that a locally integrable function $w\geq0$ belongs to the Muckenhoupt $A_{p}$ class if
\[
[w]_{A_{p}}:=\sup_{Q}\Big(\frac{1}{|Q|}\int_{Q}w\Big)\Big(\frac{1}{|Q|}\int_{Q}w^{1-p'}\Big)^{p-1}<\infty,
\]
where $p'$ is such that $\frac{1}{p}+\frac{1}{p'}=1$.
We call $[w]_{A_{p}}$ the $A_{p}$ constant or characteristic. If
$p=1$ we say that $w\in A_{1}$ if there exists a constant $\kappa>0$
such that
\begin{equation}\label{eq:CondA1}
Mw(x)\leq\kappa w(x)\qquad\text{a.e. }x\in\mathbb{R}^{n}.
\end{equation}
We define the $A_{1}$ constant or characteristic $[w]_{A_{1}}$ as
the infimum of all $\kappa$ such that \eqref{eq:CondA1} holds. It
is also a well known fact that the $A_{p}$ classes are increasing,
namely that $p\leq q\Rightarrow A_{p}\subset A_{q}.$ We can define
in a natural way the $A_{\infty}$ class as $A_{\infty}=\bigcup_{p\geq1}A_{p}$.
Associated to this $A_{\infty}$ class it is also possible to define
an $A_{\infty}$ constant as
\[
[w]_{A_{\infty}}:=\sup_{Q}\frac{1}{w(Q)}\int_{Q}M(w\chi_{Q})dx.
\]
This constant was essentially introduced by N. Fujii in \cite{F}
and rediscovered by Wilson in \cite{W}.

Another basic tool for us is the following classical reverse H\"older inequality with optimal bound, as
obtained in \cite{HPAinfty} (see also \cite{HPR1}).
\begin{Lemma} 
\label{Lem:RHI}Let $w\in A_{\infty}$. There exists $\tau_{n}>0$
such that for every $\delta\in\left[0,\frac{1}{\tau_{n}[w]_{A_{\infty}}}\right]$
and every cube $Q$
\[
\left(\frac{1}{|Q|}\int_{Q}w^{1+\delta}\right)^{\frac{1}{1+\delta}}\leq \frac{2}{|Q|}\int_{Q}w.
\]
\end{Lemma}

Finally, we will also use a variant of Rubio de Francia algorithm (see \cite[Section 5]{GCRdF} for the original algorithm).

\begin{Lemma}[\cite{CUMP,LOP2}]
\label{lem:RdF}
Denote $S(h)=v^{-\frac 1p} M(h v^{\frac 1p})$, where $v$ is a weight and $1<p<\infty$. Define a new operator R by
\[
R(h)=\sum_{k=0}^{\infty}\frac{1}{2^{k}}\frac{S^{k}h}{\|S\|_{L^{p}(v)}^{k}}.
\]
Then, for every $h\in L^p(v)$, this operator has the following properties:
\begin{enumerate}
\item $0\leq h\leq R(h)$,
\item $\|R(h)\|_{L^{p}(v)}\leq2\|h\|_{L^{p}(v)}$,
\item $R(h)v^{\frac{1}{p}}\in A_{1}$ with $\big[R(h)v^{\frac{1}{p}}\big]_{A_{1}}\leq c_{n}p'$. Furthermore, when $v=M_A w$ for some Young function $A$, we also have that $[Rh]_{A_{\infty}}\leq c_n [Rh]_{A_{3}}\leq c_{n}p'$.
\end{enumerate}
\end{Lemma}

\subsection{A sparse domination result}
\label{subsec:known}
We present here a pointwise estimate recently obtained in \cite{CACDPO}. First we recall
that a family $\mathcal{S}$ contained in a dyadic lattice $\mathcal{D}$
is a $\eta$-sparse family ($0<\eta<1$) if for each $Q\in\mathcal{S}$
there exists $E_{Q}$ such that
\begin{enumerate}
\item $\eta|Q|\leq|E_{Q}|$.
\item The sets $E_{Q}$ are pairwise disjoint.
\end{enumerate}
For a more detailed account about dyadic lattices and sparse families
we remit to \cite{LN}. Now we are in the position to state the result we borrow from \cite{CACDPO}.

\begin{Theorem}[{\cite[Theorems A and B]{CACDPO}}]
\label{Thm:Sparse}Let $T$ be defined as in \eqref{eq:TOmega} or \eqref{eq:BR}. Then for all $1<p<\infty$, $f\in L^{p}(\mathbb{R}^{n})$ and $g\in L^{p'}(\mathbb{R}^{n})$,
we have that
\[
\Big|\int_{\mathbb{R}^{n}}T(f)gdx\Big|\leq c_n C_Ts'\sup_{\mathcal{S}}\sum_{Q\in\mathcal{S}}\Big(\int_{Q}|f|\Big)\Big(\frac{1}{|Q|}\int_{Q}|g|^{s}\Big)^{1/s},
\]
where each $\mathcal{S}$ is a sparse family of a dyadic lattice $\mathcal{D}$,
\[
\begin{cases}
1<s<\infty & \text{if }T=B_{(n-1)/2}\text{ or }T=T_{\Omega}\text{ with }\Omega\in L^{\infty}(\mathbb{S}^{n-1})\\
q' \le s<\infty& \text{if }T=T_{\Omega}\text{ with }\Omega\in L^{q,1}\log L(\mathbb{S}^{n-1})
\end{cases}
\] and
\begin{equation}
\label{eq:CT}
C_{T}=\begin{cases}
\|\Omega\|_{L^{\infty}(\mathbb{S}^{n-1}),} & \text{if }T=T_{\Omega}\text{ with }\Omega\in L^{\infty}(\mathbb{S}^{n-1})\\
\|\Omega\|_{L^{q,1}\log L(\mathbb{S}^{n-1})} & \text{if } \Omega\in L^{q,1}\log L(\mathbb{S}^{n-1})\\
1 & \text{if }T=B_{(n-1)/2}.
\end{cases}
\end{equation}
\end{Theorem}

We remark the same bilinear form also applies to the case of maximally truncated oscillatory singular integrals, see \cite{KL}.

\subsection{Young functions and related maximal functions}
\label{subsec:Young}
We recall that a Young function is a convex, strictly increasing function
$A:[0,\infty)\rightarrow[0,\infty)$ such that $A(0)=0$ and $A(t)\rightarrow\infty$
as $t\rightarrow\infty$. It's clear from the definition that $A^{-1}(t)$
is well defined and is also increasing. A Young function $A$ is said to be doubling if there exists a positive constant $C$ such that $A(2t)\le C A(t)$.

For each Young function we can define its complementary function
\[
\bar{A}(s)=\sup_{t>0}\left\{ st-A(t)\right\}, \quad s\geq0.
\]
We observe that $\bar{A}$ is finite-valued if and only if $\lim_{t\rightarrow\infty}\frac{A(t)}{t}=\sup_{t>0}\frac{A(t)}{t}=\infty$,
but this will be the case for all the Young functions we are going
to deal with. We also know that $\bar{A}$ is strictly increasing
if and only if $\lim_{t\rightarrow0}\frac{A(t)}{t}=\inf_{t>0}\frac{A(t)}{t}=0$.
In that case, which will be also the case of all the explicit examples
we will introduce, $\bar{A}$ is also a Young function and enjoys
the following properties
\[
st\leq A(t)+\bar{A}(s),\quad t,s\geq0,
\]
and
\[
t\leq A^{-1}(t)\bar{A}^{-1}(t)\leq2t,\quad t>0.
\]
Associated to a Young function $A$, or more generally to any positive
function $A$, we can define the $A$-norm of a function $f$ over
a cube $Q$ as
\[
\|f\|_{A(L),Q}=\|f\|_{A,Q}:=\inf\left\{ \lambda>0\,:\,\frac{1}{|Q|}\int_{Q}A\left(\frac{|f(x)|}{\lambda}\right)dx\leq1\right\} .
\]
Each Young function and its complementary
function satisfy the following generalized H\"ošlder inequality
\[
\frac{1}{|Q|}\int_{Q}|fg|dx\leq2\|f\|_{A,Q}\|g\|_{\bar{A},Q}.
\]
We can also define in a natural way the corresponding maximal operators,
namely, given a Young function, we define the maximal operator
\[
M_{A(L)}f(x)=M_{A}f(x):=\sup_{Q\ni x}\|f\|_{A,Q}.
\]
In the case $M_{L^r}$ with $r>0$ we will keep the standard notation $M_r$. The $L^{p}$ boundedness of the maximal operators we have just defined was thoroughly studied and characterized in \cite{P}. Here we state some precise versions of sufficient conditions for the $L^{p}$ boundedness of such operators that were obtained in \cite{HP}.

Let $1<p<\infty$. A doubling Young function $A$ satisfies the $B_p$ condition if there is a positive constant $c$ such that
\begin{equation*}
\beta_p(A):=\int_{c}^{\infty}\frac{A(t)}{t^{p}}\frac{dt}{t}\approx\int_{c}^{\infty}\Big(\frac{t^{p'}}{\bar{A}(t)}\Big)^{p-1}\frac{dt}{t}<\infty.
\end{equation*}
In such case, we will say that $A\in B_p$.

\begin{Lemma}[{\cite[Lemmas 2.1 and 2.2]{HP}}]
\label{lem:Young}
Let $A$ be a Young function. Then
\[
\|M_{A}\|_{L^p}\leq c_{n}\beta_{p}(A).
\]
\end{Lemma}

By using Lemma \ref{lem:Young}, it was established in \cite{HP} that, for $A(t)=t^{p}(1+\log^{+}t)^{p-1+\delta}$ with $1<p<\infty$ and $0<\delta\leq 1$,
\begin{equation}
\|M_{\bar{A}}\|_{L^{p'}}\leq c_{n}p^{2}\left(\frac{1}{\delta}\right)^{\frac{1}{p'}}.\label{eq:MaBdd}
\end{equation}
We observe also that by standard computations we have, for $A(t)=t^{pr}$ with $1<p,r<\infty$, that
\begin{equation}\label{eq:Comptpr}
\bar{A}(t)=t^{(rp)'}\left(\frac{1}{rp}\right)^\frac{1}{rp-1}\left(1-\frac{1}{rp}\right)\leq t^{(rp)'}.
\end{equation}
Therefore $M_{\bar{A}}\leq M_{(rp)'}$. Again standard computations show that
\begin{equation}\label{eq:Comptprbdd}
\|M_{(rp)'}\|_{L^{p'}}\leq c_n p (r')^\frac{1}{p'}.
\end{equation}

For more details about Young functions and other related topics we encourage the reader to consult the classical book by M. M. Rao and Z. D. Ren \cite{RR}.

\section{Proof of Theorem \ref{ThmTfMf}}\label{Sec:ProofThmTfMf}

As explained in Section \ref{sec:intro}, we prove Theorem \ref{ThmTfMf} for $p=1$ and $p>1$ separately.

We deal with the case of $p=1$ first. Since  $w\in A_{\infty}$ we use the reverse H\"older inequality property (Lemma \ref{Lem:RHI}).  Hence if $s=1+\frac{1}{\tau_n[w]_{A_\infty}}$, then
\[
\left(\frac{1}{|Q|}\int_{Q}w^{s}\right)^{\frac{1}{s}}\leq\frac{2}{|Q|}\int_{Q}w.
\]
Thus  we have that $s'\simeq [w]_{A_{\infty}}$ and $w\in L_{loc}^s(\mathbb{R}^{n}) $. Now we let $g_R=w\chi_{Q_R}$  where $Q_R$ is the cube centered at $0$ with sidelength $R$. Then  $g_R \in L^s(\mathbb{R}^{n}) $  and hence if $f$ is smooth $|\langle Tf, g_R\rangle| <\infty$ by H\"older's inequality and the boundedness of $T$ in any $L^q$, $q \in (1,\infty)$. Taking into account these facts, and after applying first Theorem \ref{Thm:Sparse}, we have
 \[
|\langle Tf, g_R\rangle|\leq  c_{T}\, s'\,\sum_{Q\in \mathcal S} |Q|\langle f \rangle_Q\langle g_R\rangle_{s, Q}.
\]
\[
\leq c_{T}\, s'\,\sum_{Q\in \mathcal S} |Q|\langle f \rangle_Q\langle w\rangle_{s, Q}.
\]
\[
\leq 2c_{T}\, [w]_{A_{\infty}}\,\sum_{Q\in \mathcal S} \langle f \rangle_Q  \,w(Q).
\]
We are now in position to use the Carleson embedding type argument as in \cite[Lemma 4.1]{HP}, hence
 \[
|\langle Tf, g_R\rangle|\leq  c_{T}\, [w]_{A_{\infty}}^2 \, \|M f\|_{L^{1}(w)}.
\]
To conclude we just let $R \to \infty$ recalling that by assumption the left-hand side is finite, namely $\|Tf\|_{L^{1}(w)} <\infty$. All in all, we have proved
$$
\|Tf\|_{L^{1}(w)}
\leq
c_T   [w]_{A_{\infty}}^2\, \|M f\|_{L^{1}(w)}.
$$

Now for $p>1$. Observe that  $C_c^\infty$  is dense in $L^{p'}(w)$, for $w\in A_\infty$. Moreover, given $g\in C_c^\infty$, we have that $gw\chi_{w\le R}\in L^{p'}$, where $\chi_{w\le R}:=\{x:w(x)\le R\}$. By the sparse domination formula in Theorem \ref{Thm:Sparse}, we get
\begin{align*}
\left|\langle Tf, gw\chi_{w\le R}\rangle \right|\le c_T s' \sum_{Q\in \mathcal S}\langle |f| \rangle_Q \langle |gw|^s \rangle_Q^{\frac 1s} |Q|.
\end{align*}
Then, H\"older's inequality yields
\[
\langle |gw|^s \rangle_Q^{\frac 1s} \le \langle |g|^{sr} w \rangle_Q^{\frac 1{sr}} \langle w^{(s-\frac 1r)r'} \rangle_Q^{\frac 1{sr'}}.
\]
Let
\begin{equation}\label{eq:sr}
s= 1+ \frac{1}{8p\tau_n[w]_{A_\infty}},\quad r= 1+ \frac {1}{4p}.
\end{equation}
Then it is easy to check that
\[
sr<1+\frac{1}{2p}<p',\quad \mbox{and}\,\, (s-\frac 1r)r'= s+\frac{s-1}{r-1}< 1+ \frac 1{\tau_n [w]_{A_\infty}}.
\]
By monotonicity convergence theorem we can assume that for any cube $Q\in \mathcal S$, $\ell (Q)\le 2^N$ for some $N\in \mathbb N$. Then combining the arguments above we obtain
\begin{align*}
\left|\langle Tf, gw\chi_{w\le R}\rangle \right|&\le c_{p,T} [w]_{A_\infty}\sum_{Q\in \mathcal S}\langle |f| \rangle_Q \langle |g|^{sr} w \rangle_Q^{\frac 1{sr}} \langle w \rangle_Q^{1-\frac 1{sr}}|Q|\\
&= c_{p,T} [w]_{A_\infty}\sum_{Q\in \mathcal S}\langle |f| \rangle_Q  \Big( \frac 1{w(Q)}\int_Q |g|^{sr} wdx\Big)^{\frac 1{sr}} w(Q)\\
&\le c_{p,T} [w]_{A_\infty}\sum_{F\in \mathcal F}\langle |f| \rangle_F  \Big( \frac 1{w(F)}\int_F |g|^{sr} wdx\Big)^{\frac 1{sr}} \sum_{\substack{Q\in \mathcal S\\ \pi(Q)=F}}w(Q)\\
&\le c_{p,T} [w]_{A_\infty}^2\sum_{F\in \mathcal F}\langle |f| \rangle_F  \Big( \frac 1{w(F)}\int_F |g|^{sr} wdx\Big)^{\frac 1{sr}} w(F)\\
&= c_{p,T} [w]_{A_\infty}^2 \int \sum_{F\in \mathcal F}\langle |f| \rangle_F  \Big( \frac 1{w(F)}\int_F |g|^{sr} wdx\Big)^{\frac 1{sr}} \chi_F(x) wdx\\
&\le c_{p,T} [w]_{A_\infty}^2\int_{\mathbb R^n} M(f) M_{sr}^{w}(g) wdx\\
&\le c_{p,T} [w]_{A_\infty}^2 \|Mf\|_{L^{p}(w)} \|g\|_{L^{p'}(w)},
\end{align*}
where $\mathcal F$ is the family of the principal cubes in the usual sense, namely,
\[\F=\bigcup_{k=0}^\infty\F_k\] with $\F_0:=\left\{\text{maximal cubes in } \mathcal{S}\right\}$ and
\[\F_{k+1}:=\bigcup_{F\in\F_k}\text{ch}_\F(F),\quad \text{ch}_\F(F)_=\left\{Q\subsetneq F \, \text{maximal s.t. }\tau(Q) > 2\tau(F) \right\}\]where $\tau(Q)=\langle |f| \rangle_Q  \Big( \frac 1{w(Q)}\int_Q |g|^{sr} wdx\Big)^{\frac 1{sr}}$ and $\pi(Q)$ is the minimal principal cube which contains $Q$. Since we have assumed that $\|Tf\|_{L^p(w)}$ is finite, then $\langle |Tf|, |g|w\rangle $ is also finite, by dominated convergence theorem. Thus, we conclude that
\[
\left|\langle Tf, gw \rangle \right|\le c_{p,T} [w]_{A_\infty}^2 \|Mf\|_{L^{p}(w)} \|g\|_{L^{p'}(w)}.
\]
Finally by taking the supremum over $\|g\|_{L^{p'}(w)}=1$ we complete the proof.

\section{Proof of Theorem \ref{Thm:improvedAp}}
\label{sec:improvedAp}

We begin observing that Theorem \ref{Thm:Sparse} with $s=1+\varepsilon$ yields \[
|\langle Tf, g\rangle|\le  \frac {c_{n,T}}{\varepsilon}\sum_{Q\in \mathcal S} |Q|\langle f \rangle_Q\langle g\rangle_{1+\varepsilon, Q}.
\]
By the arguments in \cite[Theorem 1.2]{L}, we can obtain
\[
|\langle Tf, g\rangle|\le \frac{c_{n,T,p}}\varepsilon [v]_{A_r}^{\frac 1{1+\varepsilon}-\frac 1{p'}}([u]_{A_\infty}^{\frac 1p}+ [v]_{A_\infty}^{\frac 1{p'}})\|f\|_{L^p(w)}\|g\|_{L^{p'}(\sigma)},
\]
where
\begin{align*}
r&= \left( \frac{(1+\varepsilon)'}{p}\right)' (p-1)+1=p+ \frac{\varepsilon p}{p'-(1+\varepsilon)}\\
v&=\sigma^{\frac{1+\varepsilon}{1+\varepsilon-p'}}=w^{1+\frac {\varepsilon p'}{p'-(1+\varepsilon)}}, \quad u=w^{\frac 1{1-p}}=\sigma.
\end{align*}
By definition,
\begin{align*}
[v]_{A_r}^{\frac 1{1+\varepsilon}-\frac 1{p'}}&= \sup_Q \Big(\frac 1{|Q|}\int_Q w^{1+\frac {\varepsilon p'}{p'-(1+\varepsilon)}}\Big)^{\frac 1{1+\varepsilon}-\frac 1{p'}} \Big(\frac 1{|Q|}\int_Q \sigma  \Big)^{(r-1)(\frac 1{1+\varepsilon}-\frac 1{p'})}\\
&= \sup_Q \Big(\frac 1{|Q|}\int_Q w^{1+\frac {\varepsilon p'}{p'-(1+\varepsilon)}}\Big)^{\frac 1p \frac 1{1+ \frac {\varepsilon p'}{p'-(1+\varepsilon)}}} \Big(\frac 1{|Q|}\int_Q \sigma  \Big)^{\frac 1{p'}}.
\end{align*}
By Lemma \ref{Lem:RHI}, let
\[
\frac {\varepsilon p'}{p'-(1+\varepsilon)}=\frac 1{\tau_n [w]_{A_\infty}}.
\]
Then
\begin{align*}
[v]_{A_r}^{\frac 1{1+\varepsilon}-\frac 1{p'}}&\le 2 [w]_{A_p}^{\frac 1p}, [v]_{A_\infty}\le c_n [w]_{A_\infty}.
\end{align*}
Altogether,
\[
|\langle Tf, g\rangle| \le c_{n,T} [w]_{A_p}^{\frac 1p}[w]_{A_\infty} ([w]_{A_\infty}^{\frac 1{p'}}+[\sigma]_{A_\infty}^{\frac 1p})\|f\|_{L^p(w)}\|g\|_{L^{p'}(\sigma)}.
\]
The above estimate implies that
\[
\|T(f)\|_{L^p(w)}\le c_{n,T} [w]_{A_p}^{\frac 1p}[w]_{A_\infty} ([w]_{A_\infty}^{\frac 1{p'}}+[\sigma]_{A_\infty}^{\frac 1p})\|f\|_{L^p(w)}.
\]
Since $T$ is essentially a self-dual operator (observe that $T^t$ is associated to the kernel $\tilde{\Omega}(x):=\Omega(-x)$), by duality, we have
\begin{align*}
\|T\|_{L^p(w)}&= \|T^t\|_{L^{p'}(\sigma)}\le c_{n,T} [\sigma]_{A_{p'}}^{\frac 1{p'}}[\sigma]_{A_\infty} ([w]_{A_\infty}^{\frac 1{p'}}+[\sigma]_{A_\infty}^{\frac 1p})\\
&=c_{n,T} [w]_{A_p}^{\frac 1p}[\sigma]_{A_\infty} ([w]_{A_\infty}^{\frac 1{p'}}+[\sigma]_{A_\infty}^{\frac 1p}).
\end{align*}
Thus altogether, we obtain
\begin{align*}
\|T\|_{L^p(w)}&\le c_{n,p,T}[w]_{A_p}^{\frac 1p}([w]_{A_p}^{\frac 1{p'}}+[\sigma]_{A_\infty}^{\frac 1p}) \min\{[\sigma]_{A_\infty}, [w]_{A_\infty}\}\\
&\le c_{n,p,T}[w]_{A_p}^{\frac p{p-1}}.
\end{align*}

Now let us consider the weak type inequality. It is enough to consider the case $1<p<2$.  By the sparse domination formula in Theorem \ref{Thm:Sparse}, we get
\begin{align*}
|\langle Tf, gw\rangle|\le c_T s' \sum_{Q\in \mathcal S}\langle |f| \rangle_Q \langle |gw|^s \rangle_Q^{\frac 1s} |Q|.
\end{align*}
By similar arguments as that in Theorem \ref{ThmTfMf} yields
\begin{align*}
|\langle Tf, gw\rangle |&\le c_{p,T} [w]_{A_\infty}\sum_{Q\in \mathcal S}\langle |f| \rangle_Q \langle |g|^{sr} w \rangle_Q^{\frac 1{sr}} \langle w \rangle_Q^{1-\frac 1{sr}}|Q|\\
&= c_{p,T} [w]_{A_\infty}\sum_{Q\in \mathcal S}\langle |f| \rangle_Q  \Big( \frac 1{w(Q)}\int_Q |g|^{sr} wdx\Big)^{\frac 1{sr}} w(Q)\\
&\leq c_{p,T} [w]_{A_\infty}\int_{\mathbb{R}^n}\sum_{Q\in \mathcal S} \langle |f| \rangle_Q\chi_Q(M_{sr}^wg)wdx\\
&\leq c_{p,T} [w]_{A_\infty}\big\|\sum_{Q\in \mathcal S} \langle |f| \rangle_Q\chi_Q\big\|_{L^{p,\infty}(w)} \|M^w_{sr}g\|_{L^{p',1}(w)},
\end{align*} where the value of $s$ and $r$ are defined in \eqref{eq:sr}.
Thus, using that the sparse operators are of weak type $(p,p)$ with respect to $w\in A_p$ with constant bounded by a universal mutiple of
$[w]_{A_\infty}^{\frac{1}{p'}}[w]_{A_p}^{\frac{1}{p}}$ when $p>1$, we have
\[
|\langle Tf, gw \rangle |\le c_{p,T} [w]_{A_\infty}^{1+\frac{1}{p'}}[w]_{A_p}^{\frac{1}{p}} \,\|f\|_{L^{p}(w)} \|g\|_{L^{p',1}(w)},
\]
since one can show that  $M^\mu_{t}: L^{q,1}(\mu) \to  L^{q,1}(\mu) $, $t\geq 1$, $t<q<\infty$ with norm bounded by a dimensional multiple of $(\frac{q}{t})'$. Here $M^\mu_{t}f= M^{\mu}(f^t)^{1/t}$ and $M^{\mu}$ is the maximal function with respect to the measure $\mu$ and attached to the dyadic lattice that contains $\mathcal{S}$.  To prove the bound $(\frac{q}{t})'$ we can use the well known fact that it is enough to restrict to characteristics  of sets when dealing with the Lorentz space $L^{q,1}$, $q>1$. An application of the usual weak type $(1,1)$ property of $M^\mu$ yields immediately the bound. Another argument can be found in  \cite[Proposition A.1]{CMP-IMRN}.
In our case from that estimate it follows that $\|M^w_{sr}\|_{L^{p',1}(w)} \lesssim (\frac{p'}{sr})' < 2p'$ since  $p<2$. Indeed, $(\frac{p'}{sr})' < 2p'$ is equivalent to $sr +\frac12 < p'$, but this follows by \eqref{eq:sr} since   $sr +\frac12 < 1+\frac12+\frac12=2<p'$.

Finally, by taking supremum over $\|g\|_{L^{p',1}(w)}=1$ we have that
\[\|Tf\|_{L^{p,\infty}(w)}\leq c_{n,p}\|\Omega\|_{L^\infty(\mathbb{S}^{n-1})}[w]_{A_\infty}^{1+\frac{1}{p'}}[w]_{A_p}^{\frac{1}{p}} \|f\|_{L^{p}(w)}\leq c_{n,p}\|\Omega\|_{L^\infty(\mathbb{S}^{n-1})}[w]_{A_p}^2\|f\|_{L^p(w)}.
\]
Taking into account the strong type estimate we deduce the announced weak type estimate.

\section{Proof of Theorem \ref{Thm:A1} and Corollaries \ref{Cor:Conj} and \ref{Cor:A1}}
\label{sec:ThmA1}

We begin with the proof of \eqref{eq:MA}. We follow ideas from \cite{LOP,LOP2, HP} combined with the pointwise estimate in Theorem \ref{Thm:Sparse}. Since $T$ is essentially
a self-dual operator, if we call $A_p(t)=A(t^{1/p})$ then, by duality, it suffices to prove the following estimate
\begin{equation*}
\left\Vert \frac{Tf}{M_{A_p}w}\right\Vert _{L^{p'}(M_{A_p}w)}\leq c(p')^{2}\|M_{\bar{A}}\|_{L^{p'}} \left\Vert \frac{f}{w}\right\Vert _{L^{p'}(w)}.
\end{equation*}
Let us denote $v:=M_{A_p}w$. We compute the norm of the left-hand side by duality. Indeed, by the duality of $C_c^{\infty}(\mathbb{R}^n)$ in weighted $L^p$ spaces
we have that
\begin{equation*}
\left\Vert \frac{Tf}{v}\right\Vert _{L^{p'}(v)}=\sup_{\|h\|_{L^{p}(v)}=1}\Big|\int_{\mathbb{R}^{n}} Tf(x)h(x) dx\Big|=\sup_{\substack{h\in C_c^{\infty}(\mathbb{R}^n)\\ \|h\|_{L^{p}(v)=1}}}\Big|\int_{\mathbb{R}^{n}}Tf(x)h(x)dx\Big|.
\end{equation*}
 We define operators $S(h)$ and $R(h)$ as in Lemma \ref{lem:RdF} (observe that, since $h\in C_c^{\infty}$, then $h\in L^{p'}(\mathbb{R}^n)$). Then, using Theorem \ref{Thm:Sparse} and the first property of the operator $R$ in Lemma \ref{lem:RdF} we have that
\begin{align}
\label{eq:1}
\notag\Big|\int_{\mathbb{R}^{n}}T(f)h dx\Big|&\leq c_{n,T}s'\sup_{\mathcal{S}}\sum_{Q\in\mathcal{S}}\Big(\int_{Q}|f|\Big)\Big(\frac{1}{|Q|}\int_{Q}h^{s}\Big)^{1/s}\\
&\leq c_{n,T}s'\sup_{\mathcal{S}}\sum_{Q\in\mathcal{S}}\Big(\int_{Q}|f|\Big)\Big(\frac{1}{|Q|}\int_{Q}(Rh)^{s}\Big)^{1/s}
\end{align}
with $1<s<\infty$ to be chosen. Hence, it suffices to control
$$\sum_{Q\in\mathcal{S}}\big(\int_{Q}|f|\big)\big(\frac{1}{|Q|}\int_{Q}(Rh)^{s}\big)^{1/s}$$
for every sparse family $\mathcal{S}$. To do this we are
going to use the reverse H\"older inequality, namely, Lemma \ref{Lem:RHI}.
We choose $s=1+\frac{1}{\tau_{n}[Rh]_{A_{\infty}}}$ so that $s'\simeq[Rh]_{A_{\infty}}\leq c_{n}p'$.
Then, by reverse H\"older inequality, we get
\begin{equation}
\label{eq:2}
\sum_{Q\in\mathcal{S}}\Big(\int_{Q}|f|\Big)\Big(\frac{1}{|Q|}\int_{Q}(Rh)^{s}\Big)^{1/s}\leq2\sum_{Q\in\mathcal{S}}\int_{Q}|f|\frac{1}{|Q|}\int_{Q}Rh=2\sum_{Q\in\mathcal{S}}\frac{1}{|Q|}\int_{Q}|f||Rh(Q).
\end{equation}
Using \cite[Lemma 4.1]{HP} with the weight $w=Rh$, we have that
\begin{equation}
\label{eq:3}
\sum_{Q\in\mathcal{S}}\frac{1}{|Q|}\int_{Q}|f||Rh(Q)\leq c_{n}[Rh]_{A_{\infty}}\|Mf\|_{L^{1}(Rh)}\leq c_{n}p'\|Mf\|_{L^{1}(Rh)}.
\end{equation}
From this point, by H\"older's inequality and the second property of the operator $R$ in Lemma \ref{lem:RdF},
\begin{equation}
\label{eq:4}
\left\Vert Mf\right\Vert _{L^{1}(Rh)}\leq\left(\int_{\mathbb{R}^{n}}(Mf)^{p'}(v)^{1-p'}\right)^{\frac{1}{p'}}\left(\int_{\mathbb{R}^{n}}(Rh)^{p}v\right)^{\frac{1}{p}}\leq2\left\Vert \frac{Mf}{v}\right\Vert _{L^{p'}(v)}.
\end{equation}
Hence, combining estimates \eqref{eq:1}, \eqref{eq:2}, \eqref{eq:3}, and \eqref{eq:4},
we have that
\begin{equation*}
\Big\| \frac{Tf}{v}\Big\| _{L^{p'}(v)}\leq c(p')^{2} \Big\|  \frac{Mf}{v}\Big\|_{L^{p'}(v)}.
\end{equation*}
Let us recover the initial notation for $v:=M_{A_p}w$. To end the proof of \eqref{eq:MA}, we have to prove that
\begin{equation}
\label{eq:lop}\Big\| \frac{Mf}{M_{A_p}w}\Big\| _{L^{p'}(M_{A_p}w)}\leq c \|M_{\bar{A}}\|_{L^{p'}}\Big\| \frac{f}{w}\Big\| _{L^{p'}(w)} \end{equation}
which in turn is equivalent to prove that
\[ \|  M(fw)\| _{L^{p'}((M_{A_p}w)^{1-p'})}\leq c \|M_{\bar{A}}\|_{L^{p'}}\Big\|  f\Big\| _{L^{p'}(w)} \]
but this inequality was obtained in \cite[pp. 618--619]{HP}. So this ends the proof of \eqref{eq:MA}.

If we choose $A(t)=t^p(1+\log^+t)^{p-1+\delta}$ with $\delta>0$, since we know that
\[\|M_{\bar{A}}\|_{L^{p'}}\leq c_n p^2\left(\frac{1}{\delta}\right)^\frac{1}{p'},\]
this yields \eqref{eq:LlogL}, which was stated to be sharp in \cite{HP}.
If we choose $A(t)=t^{pr}$ we know that, taking into account \eqref{eq:Comptpr},  $M_{\bar{A}}\leq M_{(rp)'}$. Now recalling \eqref{eq:Comptprbdd} and applying \eqref{eq:MA} for  $A(t)=t^{pr}$, we obtain \eqref{eq:Mr}.
If we assume that $w\in A_{\infty}$, choosing $r=1+\frac{1}{\tau_n[w]_{A_\infty}}$ in \eqref{eq:Mr}
we have that $r'\simeq[w]_{A_{\infty}}$ and it readily follows from the reverse H\"olˆder inequality (Lemma \ref{Lem:RHI}) that $M_{r}w\leq 2Mw$ for every $x\in\mathbb{R}^{n}$.
This yields  \eqref{eq:M}. Furthermore, if $w\in A_{1}$, from (\ref{eq:M})
and the definition of the $A_{1}$ constant, we obtain \eqref{eq:A1}. This finishes the proofs of Theorem \ref{Thm:A1} and Corollaries \ref{Cor:Conj} and \ref{Cor:A1}

\section{Proof of Theorem \ref{Thm:weak11}}
\label{sec:Thmweak11}
In this section we shall give a proof for Theorem \ref{Thm:weak11}. We start with $T=T_{\Omega}$.
To study the weighted weak $(1,1)$ bound, one needs to estimate the constant in the following inequality:
\begin{align*}
\sup_{\alpha>0} \alpha w(\{x\in \bbR^n: |T_\Omega(f)(x)|> \alpha \}) \le C_w \|f\|_{L^1(w)}.
\end{align*}
To this end, we need to use some estimates obtained by Seeger \cite{S}. Denote
\[
K_j(x)= K(x) (\phi(2^{-j+1}|x|)- \phi(2^{-j+2}|x|)),
\]
where $\phi\in C^\infty((0,\infty))$ satisfying $\phi(t)=1$ when $t\le 1$ and $\phi(t)=0$ when $t\ge 2$.  Then it is obvious that
\begin{equation}\label{eq:s1}
\supp K_j \subset \{x: 2^{j-2}\le |x|\le 2^{j}\},
\end{equation}
and
\begin{equation}\label{eq:s2}
\sup_{0\le \ell \le N} \sup_j r^{n+\ell} \left| \left(\frac \partial {\partial r}\right)^\ell K_j(r\theta) \right|\le C_{N,n} \|\Omega\|_{L^\infty}.
\end{equation}
Given $\alpha>0$, without loss of generality we assume $f \ge 0$ and we form the Calder\'on-Zygmund decomposition of $f$ at height $\alpha/ {\|\Omega\|_{L^\infty}}$. In this way, there is a collection  of non-overlapping dyadic cubes $\{Q\}$ such that $f=g+b$, where $ \frac{\alpha}{\|\Omega\|_{L^\infty}} < \langle f \rangle_Q \le \frac{2^n\alpha}{\|\Omega\|_{L^\infty}} $ and, for the good part,
\[
0\le g\le \frac{2^n \alpha}{\|\Omega\|_{L^\infty}},
\]
whereas, for the bad part,
\[
 b= \sum_{Q} b_Q =\sum_j \sum_{Q: \ell(Q)= 2^j}b_Q=: \sum_j B_j,
\]
and moreover,
\[
\supp b_Q \subset Q, \quad \text{and} \quad \|b_Q\|_{L^1}\le \frac{2^{n+1} \alpha}{\|\Omega\|_{L^\infty}}  |Q|.
\]
Then
\begin{align*}
&w(\{x\in \bbR^n: |T_{\Omega}f(x)|> \alpha \}) \\ &\quad \le w\left(\left\{x\notin E: |T_{\Omega}g(x)|> \frac\alpha {2}\right\}\right)
+ w\left(\left\{x\notin E: |T_{\Omega}b(x)|>\frac \alpha {2}\right\}\right)\\
&\qquad+w(E)
\\&\quad=:I+II+w(E),
\end{align*}where $E:= \cup_Q 3Q$ and we have
\begin{align*}
w(E)\le \sum_{Q} \frac {w(3Q)}{|3Q|} 3^n |Q|&\le \sum_Q 3^n[w]_{A_1} \frac{\|\Omega\|_{L^\infty}}{\alpha}\int_Q f\inf_{3Q} w(x)\\
&\le  3^n[w]_{A_1} \frac{\|\Omega\|_{L^\infty}}{\alpha} \|f\|_{L^1(w)}.
\end{align*}
It remains to estimate $I$ and $II$. For $I$, by Chebyshev inequality, estimate \eqref{eq:Mr} in Corollary \ref{Cor:A1}, the fact that $|g(x)|\le 2^n\alpha/\|\Omega\|_{L^{\infty}}$, and an argument in \cite[pp. 302--303]{P0} (see also \cite[p. 282] {CW}), we have
\begin{align*}
I &\le c_n^{p_0} \alpha^{-p_0}\int_{\mathbb{R}^n\setminus E}|T_{\Omega}g(y)|^{p_0}w(y)\,dy \\
&\le \alpha^{-p_0}(c_n \|\Omega\|_{L^\infty} p_0 (p_0')^2)^{p_0} (r')^{p_0-1}\int_{\mathbb{R}^n}|g(y)|^{p_0}M_r(w\chi_{\bbR^n \setminus E})(y)\,dy\\
&\le \alpha^{-p_0}(c_n \|\Omega\|_{L^\infty} p_0 (p_0')^2)^{p_0} (r')^{p_0-1} \frac{\alpha^{p_0-1}}{\|\Omega\|_{{L^\infty}}^{p_0-1}}\int_{\mathbb{R}^n}|g(y)|M_r(w\chi_{\bbR^n \setminus E})(y)\,dy\\
&\le \frac{c_n\|\Omega\|_{L^\infty}}{\alpha} \big(p_0 (p_0')^2\big)^{p_0} (r')^{p_0-1} \int_{\mathbb{R}^n}|f(y)|M_rw(y)\,dy\\
&\le \frac{c_n\|\Omega\|_{L^\infty}}{\alpha} \big(p_0 (p_0')^2\big)^{p_0}  (r')^{p_0-1}[w]_{A_1}\|f\|_{L^1(w)}\\
&\le \frac{c_n\|\Omega\|_{L^\infty}}{\alpha} [w]_{A_1} (\log([w]_{A_\infty}+1))^2\|f\|_{L^1(w)},
\end{align*}
where in the last step, we have chosen $p_0= 1+\frac 1{\log([w]_{A_\infty}+1)}$ and $r=1+\frac{1}{\tau_n[w]_{A_\infty}}$, the exponent from the optimal reverse H\"older property as in  Lemma \ref{Lem:RHI}. To estimate $II$,
by the decomposition of the kernel, for $x\notin E$ we have
\begin{align*}
T(b)(x) = \sum_{j\in \bbZ} K_j \ast \Big(\sum_{s\in \bbZ} B_{j-s}\Big)(x)=   \sum_{s\in \bbZ} \sum_{j\in \bbZ} K_j \ast B_{j-s}(x)=\sum_{s\ge 0} \sum_{j\in \bbZ} K_j \ast B_{j-s}(x).
\end{align*}
To proceed our argument, we need to use an auxiliary operator $\Gamma_j^s$ (for the precise definition, we refer the reader to \cite[pp. 97--98]{S}, we are following the same notation therein).
Since we have checked that $K_j$ satisfies \eqref{eq:s1} and \eqref{eq:s2}, then it was shown by Seeger \cite{S} that when $N$ is sufficiently large (but depends only on dimension), then there exists $\epsilon>0$ such that
\begin{equation}\label{eq:Se1}
\left\|\sum_j \Gamma_{j}^s \ast B_{j-s}\right\|_{L^2}^2 \le c_{n} 2^{-s\epsilon}\alpha\sum_{Q}\|b_Q\|_{L^1},
\end{equation}
and
\begin{equation}
\label{eq:Se11}
\left\|(K_j-\Gamma_j^s)\ast b_Q\right\|_{L^1}\le c_n 2^{-s\epsilon}\|b_Q\|_{L^1}.
\end{equation}
Indeed, inequalities \eqref{eq:Se1} and \eqref{eq:Se11} are contained essentially in \cite[Lemma 2.1]{S} and  \cite[Lemma 2.2]{S}, respectively.
The latter implies immediately that
\begin{equation}
\label{eq:Se2}
\left\|\sum_j(K_j-\Gamma_j^s)*B_{j-s}\right\|_{L^1}\le c_n \|\Omega\|_{L^\infty}2^{-s\epsilon}\sum_Q\|b_Q\|_{L^1},
\end{equation}
where $b_Q$ are the bad functions from the Calder\'on--Zygmund decomposition of $f$ described above. Let
\begin{equation*}
E_\alpha^s:= \Big\{x\notin E: \big|\sum_{j}K_j* B_{j-s}\big|>\alpha \Big\}.
\end{equation*}
Then for any $\alpha>0$, we have, by \eqref{eq:Se1} and \eqref{eq:Se2},
\begin{equation}
\label{eq:leb}
|E_\alpha^s|\le \frac{c_n \|\Omega\|_{L^\infty}}{\alpha} 2^{-s\epsilon}\sum_Q\|b_Q\|_{L^1}\le c_n 2^{-s\epsilon}\sum_Q |Q|.
\end{equation}
On the other hand, taking into account \eqref{eq:s1}, it is easy to check that
\begin{equation}
\label{eqL1}
\begin{aligned}
&\sum_j \|K_j* B_{j-s}\|_{L^1(w)}\\
&\le\sum_j\sum_{Q: \ell(Q)= 2^{j-s}} \iint |K_j(x-y)||b_Q(y)|dy w(x) dx\\
&\le \|\Omega\|_{L^\infty} \sum_j\sum_{Q: \ell(Q)= 2^{j-s}} \int|b_Q(y)| \int_{|x-y|\le 2^j} 2^{-jn}w(x) dx\,dy\\
&\le\|\Omega\|_{L^\infty}\sum_j\sum_{Q: \ell(Q)= 2^{j-s}} \int|b_Q(y)| \inf_{y'\in Q}\int_{|x-y'|\le c_n 2^{j+1}} 2^{-jn}w(x) dx\,dy\\
&\le c_{n} \|\Omega\|_{L^\infty} \sum_{Q } \|b_Q\|_{L^1}\inf_Q Mw\\
&\le c_{n}\, \alpha \sum_{Q } |Q|\inf_Q Mw.
\end{aligned}
\end{equation}
Now we are in the position to use interpolation with change of measure. We follow the strategy of \cite{FS}.
By \cite[Lemma 6]{FS}, \eqref{eq:leb} and \eqref{eqL1} imply
\begin{equation*}
\int_{E_\alpha^s} \min(w(x), u)dx\le c_n \sum_Q |Q|\min (u2^{-s\epsilon}, \inf_Q Mw).
\end{equation*}
Since, for $A>0$,
\begin{equation*}
\int_0^\infty \min (A, u)u^{-1+\theta}\frac{du}{u}=\frac 1{\theta(1-\theta)} A^\theta,
\end{equation*}
then we get
\begin{align*}
\int_{E_\alpha^s}w(x)^\theta dx&= \theta (1-\theta) \int_{E_\alpha^s}\int_0^\infty \min (w(x), u)u^{-1+\theta}\frac{du}{u} dx\\
&\le c_n \theta (1-\theta) \sum_Q |Q|\int_0^\infty \min (u2^{-s\epsilon}, \inf_Q Mw) u^{-2+\theta} du\\
&\le
c_n 2^{-s\epsilon(1-\theta)}\alpha^{-1}\|\Omega\|_{L^\infty}\int |f(x)|(Mw)^\theta dx.
\end{align*}
Rescaling the weight $w$ we obtain
\begin{equation}
\label{distribution}
w(E_\alpha^s)\le c_n 2^{-s\epsilon(1-\theta)}\alpha^{-1}\|\Omega\|_\infty\int |f(x)|(M_{1/\theta}w) dx.
\end{equation}
To get a better constant than \cite{FS}, in the last step, we shall split the summation  in two terms. For $s_0$ which will be determined later, we have
\begin{align*}
&w\Big(\Big\{x\notin E: |\sum_s\sum_j K_j*B_{j-s}|>\alpha\Big\}\Big)\\&\le
w\Big(\Big\{x\notin E: |\sum_{s=1}^{s_0}\sum_j K_j*B_{j-s}|>\frac\alpha{2}\Big\}\Big)\\
&\qquad+w\Big(\Big\{x\notin E: |\sum_{s=s_0+1}^{\infty}\sum_j K_j*B_{j-s}|>\frac\alpha{2}\Big\}\Big)\\
&\le \frac 2\alpha\sum_{s=1}^{s_0}\|\sum_j K_j*B_{j-s}\|_{L^1(w)}\\
&\quad+\sum_{s=s_0+1}^\infty w\Big(\Big\{x\notin E: |\sum_j K_j*B_{j-s}|>\frac {c\epsilon(1-\theta)\alpha}{2}2^{-(s-s_0)\epsilon(1-\theta)/3}\Big\}\Big)=:III+IV,
\end{align*}
where for the second term in the first inequality we turned $\alpha$ into $c\epsilon (1-\theta)2^{-s\epsilon(1-\theta)/3}\alpha $, with $c>0$ an absolute constant such that $c\epsilon (1-\theta)\sum_{s\ge1}2^{-s\epsilon(1-\theta)/3}=1$.
The estimate of $III$ is easy,
\[
III\le s_0c_n \|\Omega\|_{L^\infty} \alpha^{-1} \sum_{Q } \|b_Q\|_{L^1}\inf_Q Mw\le s_0c_n\|\Omega\|_{L^\infty} \alpha^{-1} [w]_{A_1}\|f\|_{L^1(w)}.
\]
To estimate $IV$, by \eqref{distribution}, we have
\begin{align*}
IV&\le \sum_{s=s_0+1}^\infty \frac {c_n}{\alpha \epsilon(1-\theta)}2^{-s_0\epsilon(1-\theta)/3}2^{-2s\epsilon(1-\theta)/3}\|\Omega\|_{L^\infty}\int |f(x)|(M_{1/\theta}w) dx\\
&\le \sum_{s=s_0+1}^\infty \frac {c_n}{\alpha \epsilon(1-\theta)}2^{-s_0\epsilon(1-\theta)}2^{-2(s-s_0)\epsilon(1-\theta)/3}\|\Omega\|_{L^\infty}\int |f(x)|(M_{1/\theta}w) dx\\
&\le  \frac {c_n}{\alpha \epsilon^2 (1-\theta)^2 }2^{-s_0\epsilon(1-\theta)}\|\Omega\|_{L^\infty}\int |f(x)|(M_{1/\theta}w) dx
\end{align*}
By the reverse H\"older inequality, one can take \[\theta\simeq \frac{c_n[w]_{A_\infty}}{1+c_n[w]_{A_\infty}}.
\]
Then
\[
(M_{1/\theta}w)(x) \le c [w]_{A_1} w(x).
\]
Since $\epsilon$ is an absolute constant, finally, we can take
\[
s_0:= \frac{1}{\epsilon(1-\theta)}\log_2([w]_{A_\infty}+1)\eqsim [w]_{A_\infty}\log_2([w]_{A_\infty}+1).
\]
Then altogether,
\begin{align*}
w\Big(\Big\{x\notin E: &|\sum_{s\ge 0}\sum_j K_j*B_{j-s}|>\alpha\Big\}\Big)\\
&\le c_n \alpha^{-1}[w]_{A_1}[w]_{A_\infty}\log_2([w]_{A_\infty}+1)\|\Omega\|_{L^\infty}  \|f\|_{L^1(w)}.
\end{align*}

It remains to study the case for $B_{(n-1)/2}$. The main difference is the estimate of the following term
\[
w(\{x\notin E: |B_{(n-1)/2}(b)(x)|>\frac \alpha 2\}).
\]
Since it's well known (see \cite[p. 340]{GM} and also \cite{C}) that the kernel of $B_{(n-1)/2}$ is of the form
\[
c_n \frac{\cos (2\pi |x|-\pi n/2 )}{|x|^n}\chi_{\{|x|\ge 1\}}+ O\big(\frac 1{1+|x|^{n+1}}\big),
\]
the error term is bounded by the maximal function pointwise, so we only need to care about the first term. Define
\[
H_j (x):=c_n\frac{\cos (2\pi |x|-n\pi/2 )}{|x|^n}(\phi(2^{-j+1 }|x|)-\phi(2^{-j+2}|x|)), \quad j\ge 1.
\]
 It is easy to check that $H_j$ still satisfies the assumption \eqref{eq:s1} and \eqref{eq:s2}, so the estimate is almost the same and we conclude the proof of Theorem \ref{Thm:weak11}.

\section{Proof of Theorem \ref{Thm:omegalq}}
\label{sec:Thmomegalq}
In this section we are concerned with the proof of Theorem \ref{Thm:omegalq}. Namely, we shall get the following inequality
\begin{equation}\label{eq:simplyq}
\|\mathcal A_{r, \mathcal S}(f)\|_{L^p(w)}\le C(p,r,A) \|f\|_{L^p(M_{A_p}w)},
\end{equation}
where $\bar A \in B_{p'}$ (see Subsection \ref{subsec:Young}) and $f\ge 0$.
We plan to use the so-called `maximal function trick' (see e.g. \cite{L}) to simplify the inequality. For simplicity, denote again $v:= M_{A_p} w$ and let $u= v^{\frac {r}{r-p}}$.  Then we can rewrite \eqref{eq:simplyq} as
\begin{align*}
\Big\|\sum_{Q\in \mathcal S} (\langle f^{r} u^{-1}\rangle_Q^{u})^{\frac 1{r}}\langle u\rangle_{Q}^{\frac 1{r}}\chi_Q\Big\|_{L^p(w)}\le C(p,r,A) \|f^{r}u^{-1}\|_{L^{p/{r}}(u)}^{\frac 1{r}}.
\end{align*}
By a change of variable, this is equivalent to
\[
\Big\|\sum_{Q\in \mathcal S} (\langle f^{r}\rangle_Q^{u})^{\frac 1{r}}\langle u\rangle_{Q}^{\frac 1{r}}\chi_Q\Big\|_{L^p(w)}\le C(p,r,A) \|f^{r}\|_{L^{p/{r}}(u)}^{\frac 1{r}}= C(p,r,A) \|f\|_{L^p(u)}.
\]
Then it suffices to prove the following inequality
\begin{equation}\label{eq:maxtrick}
\Big\|\sum_{Q\in \mathcal S}  \langle f \rangle_Q^{u}  \langle u\rangle_{Q}^{\frac 1{r}}\chi_Q\Big\|_{L^p(w)}\le \frac{ C(p,r,A)}{(p/{r})'} \|f\|_{L^p(u)}.
\end{equation}
Indeed, notice that once \eqref{eq:maxtrick} holds, then
\begin{align*}
\Big\|\sum_{Q\in \mathcal S} (\langle f^{r}\rangle_Q^{u})^{\frac 1{r}}\langle u\rangle_{Q}^{\frac 1{r}}\chi_Q\Big\|_{L^p(w)}
&\le \Big\|\sum_{Q\in \mathcal S}  \langle M_{r}^u(f)\rangle_Q\langle u\rangle_{Q}^{\frac 1{r}}\chi_Q\Big\|_{L^p(w)}\\
&\le \frac{ C(p,r,A)}{(p/{r})'} \|M_{r}^u(f)\|_{L^p(u)} \\
&\le   C(p,r,A)\|f\|_{L^p(u)},
\end{align*}
where
\[
M_{r}^u(f)(x):= \sup_{Q\ni x} (\langle |f|^{r}\rangle_Q^{u})^{\frac 1{r}}.
\]
So let us prove  \eqref{eq:maxtrick}. By a change of variable again, \eqref{eq:maxtrick} is equivalent to the following
\begin{equation}\label{eq:finalred}
\Big\|\sum_{Q\in \mathcal S}  \langle f \rangle_Q   \langle u\rangle_{Q}^{\frac 1{r}-1}\chi_Q\Big\|_{L^p(w)}\le \frac{ C(p,r,A)}{(p/{r})'} \|f\|_{L^p(u^{1-p})}.
\end{equation}
Thus now we only need to focus on \eqref{eq:finalred}. For convenience, set
\[
T_{\mathcal S}(f)= \sum_{Q\in \mathcal S}  \langle f\rangle_Q   \langle u\rangle_{Q}^{\frac 1{r}-1}\chi_Q.
\]By duality, \eqref{eq:finalred} is equivalent to
\begin{equation}\label{eq:twoweightrdf}
\|T_{\mathcal S}(f)\|_{L^{p'}(u)}\le \frac{ C(p,r,A)}{(p/{r})'} \|f\|_{L^{p'}(w^{1-p'})}.
\end{equation}
Starting from the left-hand side of \eqref{eq:twoweightrdf}, by duality again, we have
\begin{align*}
\|T_{\mathcal S}(f)\|_{L^{p'}(u)}&= \sup_{\|h\|_{L^p(u^{1-p})}=1} \int |T_{\mathcal S}(f)|\cdot |h|\\
&\le \sup_{\|h\|_{L^p(u^{1-p})}=1}\sum_{Q\in \mathcal S}  \langle f\rangle_Q   \langle u\rangle_{Q}^{\frac 1{r}-1}\int_Q |h|.
\end{align*}
Notice that since for any $0<\delta <1$, $v^\delta=(M_{A_p}(w))^{\delta}$ is an $A_1$ weight with $[v^{\delta}]_{A_1}\le \frac{c_n}{1-\delta}$, see e.g. \cite[pp. 110-113]{CUMP}.  By H\"older inequality, since $pr>p-r$, we have
\begin{align*}
\langle u\rangle_Q^{\frac 1{r}-1} &=\frac 1{\langle (v^{-1})^{\frac{r}{p-r}}\rangle_Q^{\frac 1{r'}}}= \frac 1{\langle (v^{-\frac1 p})^{\frac{pr}{p-r}}\rangle_Q^{\frac 1{r'}}}\\
&\le\frac 1{\langle v^{-\frac1p}\rangle_Q^{\frac{pr}{(p-r)r'}}} \le \langle v^{\frac1p}\rangle_Q^{\frac{pr}{(p-r)r'}}\le (c_n p')^{\frac{pr}{(p-r)r'}} \inf_{x\in Q} v(x)^{\frac {r-1}{p-r}}.
\end{align*}
Now we form the Rubio de Francia algorithm (Lemma \ref{lem:RdF}). For simplicity, set $$h_{p,r,v}(x)= |h(x)|v(x)^{\frac {r-1}{p-r}}.
$$ It is easy to check that $\|h_{p,r,v}\|_{L^p(v)}=\|h\|_{L^p(u^{1-p})}=1$.
Then define the operator $R(h_{p,r,v})$ as in Lemma \ref{lem:RdF}.
Finally we have, by the properties of $R(h_{p,r,v})$,
\begin{align*}
\sum_{Q\in \mathcal S}  \langle f\rangle_Q   \langle u\rangle_{Q}^{\frac 1{r}-1}\int_Q |h|&\le
   (c_n p')^{\frac {p(r-1)}{p-r}} \sum_{Q\in \mathcal S}  \langle f\rangle_Q \int_Q h_{p,r,v}\\&\le   (c_n p')^{\frac {p(r-1)}{p-r}}    \sum_{Q\in \mathcal S}  \langle f\rangle_Q  R(h_{p,r,v})(Q)\\
 &\le c_n p'  (c_n p')^{\frac {p(r-1)}{p-r}}  \int M(f) R(h_{p,r,v})\\
 &\le (c_n p')^{\frac {r(p-1)}{p-r}}  \| M(f)  \|_{L^{p'}(v^{1-p'})}\|R(h_{p,r,v})\|_{L^p(v)}\\
 &\le  (c_n p')^{\frac {r(p-1)}{p-r}} \|M_{\bar{A}}\|_{L^{p'}}  \|f\|_{L^{p'}(w^{1-p'})},
\end{align*}
where in the last step we have used \eqref{eq:lop}.  Altogether, we obtain
\[
\|\mathcal A_{r,\mathcal S}(f)\|_{L^p(w)}\le  (c_n p')^{\frac {r(p-1)}{p-r}} \big(\frac p{r}\big)'  \|M_{\bar{A}}\|_{L^{p'}}\|f\|_{L^p(M_{A_p}w)}.
\]

\begin{Remark}
By the result in \cite{L2}, our result applies to singular integral operators with $L^r$-H\"ormander condition as well. Specifically, our result applies to Fourier multipliers with H\"ormander condition.
\end{Remark}

\appendix

\section{} \label{A}

In this appendix we further exploit Theorem \ref{ThmTfMf} by showing a couple of two different types of examples. The first one is related to general Banach Function Spaces (BFS) $X$ as can be seen from \cite{BS}.  We follow here the theory developed in \cite{CUMP, CGMP}. For the second type of examples we consider variable $L^p$ spaces following \cite{CFMP}. In both applications  the duality plays a central role.

We will denote throughout this appendix that $T$ is either $T_\Omega$ with $\Omega \in L^\infty$ satisfying $\int_{\mathbb{S}^{n-1}}\Omega =0$ or $B_{(n-1)/2}$. The initial key inequality is the case $p=1$ from Theorem~\ref{ThmTfMf}, namely
\begin{equation} \label{p=1}
\|Tf\|_{L^{1}(w)}
\leq
c \|M f\|_{L^{1}(w)}
\end{equation}
for any $w\in A_{\infty}$ and for any smooth function such that the left-hand side is finite.  Further, from the same theorem we have a good control of the constant, $c \approx  [w]^{2}_{A_{\infty}}$. However,  we don't need to be so precise in this appendix. More important, we need the following vector-valued extension from part b) of Theorem \ref{extrapcorollary}:

\begin{equation}\label{vector-valued-p=1}
\Big\|
\Big(\sum_j |T f_j|^q\Big)^{1/q}
\Big\|_{L^1(w)}
\le
c\, \Big\| \Big(\sum_j (Mf_j)^q\Big)^{1/q} \Big\|_{L^1(w)}
\end{equation}
which holds for any $q\in(0,\infty)$ and for $w\in A_{\infty}$.

\subsection{Banach function spaces }

\begin{Theorem} \label{BFS}
Let $X$ be a BFS such that $M: X' \to X'$ where $X'$ is the associate space to $X$. Then,
\begin{itemize}
\item[a) ] {\bf Scalar context}.
\begin{equation*}
\|Tf\|_X \le c\, \|Mf\|_X,
\end{equation*}
for any smooth function such that the left-hand side is finite.

\item[b) ] {\bf Vector-valued extension}.  If $q\in (0,\infty)$ then
\begin{equation*}
\Big\|
\Big(\sum_j |T f_j|^q\Big)^{1/q}
\Big\|_{X}
\le
c\, \Big\| \Big(\sum_j (Mf_j)^q\Big)^{1/q} \Big\|_{X}
\end{equation*}
and if $q>1$  this is bounded by $c\| M(\|f\|_{\ell^q}) \|_{X}$.
\end{itemize}
\end{Theorem}

Therefore, we get the following:

\begin{Corollary}  Let $T$ be as above. Let $X$ be a
rearrangement invariant BFS such that the Boyd indices
${\overline\alpha}_{X}$ and ${\underline \alpha}_{X}$ satisfy
\[
0< {\underline \alpha}_{X} \le {\overline\alpha}_{X}   <1
\]
namely,  $M: X \to X$ \,and\, and $M: X' \to X'$, then\\
$$T: X\to X$$
and
$$T: X_{\ell^{q}} \to X_{\ell^{q}}.
$$

\end{Corollary}

The corollary follows directly from Lorentz-Shimogaki's
characterization of the rearrengement invariant BFS for which the
Hardy--Littlewood maximal is bounded as can be found in \cite{BS}.

 \begin{proof}[Proof of Theorem \ref{BFS}]

a)
The $\|Tf\|_X$ can be writen as
\[
\|Tf\|_X = \sup \Big| \int_{\bbR^n}Tf\,g \Big|
\]
where the supremum is taken over all functions $g \in X'$ with $\|g\|_{X'}=1$. Let us fix one of these $g$.
We now adapt Rubio de Francia's algorithm to this context:
consider
\[
G= \sum_{k=0}^{\infty} \frac{M^k(g)}{(2\|M\|_{X'})^{k}}
\]
where $M^k$ is the operator $M$ iterated $k$ times and $A$ is the norm of $M$ as bounded operator on $X'$. It
is immediate to see that:

a) $g\le G$

b) $\|G\|_{X'} \le 2\, \|g\|_{X'}$

c) $G \in A_1$, in fact $MG \le 2\|M\|_{X'}\, G $

In particular since $G \in A_{\infty}$ we can apply \eqref{p=1}
\[
\int_{\bbR^n} |Tf|\,|g| \le \int_{\bbR^n} |Tf|\,G
\le C\, \int_{\bbR^n} Mf\,G
\le C\, \|Mf\|_X  \|G\|_{X'}
\le C\, \|Mf\|_X  \|g\|_{X'}.
\]
Then, taking the supremum over all $g \in X'$ we deduce the theorem.

b) The proof of the first inequality it is identical to the proof of the scalar situation using \eqref{vector-valued-p=1}. For the second we will use the following pointwise estimate contained in \cite{CGMP}: there exists a constant
$c>0$ depending on $q,\delta,n$, such that
\begin{equation*}
M^\#_\delta \Big(\overline{M}_{q} f\Big)(x)
\leq
c\, M(\|f\|_{\ell^q})(x),   \qquad x  \in \mathbb R^n, q>1.
\end{equation*}
where we use the notation
$$ \M_qf(x) = \left(\sum_{i=1}^\infty Mf_j(x)^q\right)^{1/q}, $$
and
$$M^\#_\delta g(x)= M^\#(|g|^\delta)(x)^{\frac1\delta},
$$
where the Fefferman-Stein sharp maximal function is given by
$$
M^{\#}f(x)
=
\sup_{x \in B} \frac{1}{|Q|}\int_{Q} |f(y)- f_{Q}|\,dy.
$$

Assuming for the moment this result, we use the well known
Fefferman-Stein estimate \cite{J},
$$
\int_{\mathbb R^n} |f(x)|^p\, w(x)\,d x
\le
C\,\int_{\mathbb R^n} M^\# f(x)^p\, w(x)\,d x,
$$
for any $A_\infty$-weight $w$, any $p$, $0<p<\infty$ and for any
function $f$ such that left-hand side is finite.  Hence, if
$0<\delta<1$,
\begin{eqnarray*}
\int_{\bbR^n} \overline{M}_{q}f(x)^p \, w(x)\,d x
&=&
\int_{\bbR^n}
\big(\overline{M}_{q}f(x)^{\delta}\big)^{\frac{p}{\delta}}
\,
w(x)\,d x\\
&\le&
C\,\int_{\bbR^n} M^\#_\delta (\overline{M}_{q}f)(x)^p \,
w(x)\,d x\\
&\le&
C\,\int_{\bbR^n}  M(\|f\|_{\ell^q})(x)^p\, w(x)\,d x.
\end{eqnarray*}
The proof is finished.

\end{proof}

\subsection{$L^{p}$ variable theory}

\

Given  a measurable function $p: \bbR^n \rightarrow [1,\infty)$,\,
$L^{p(\cdot)}(\bbR^n)$ denotes the set of measurable functions $f$
on $\bbR^n$ such that for some $\lambda>0$,
\[ \int_{\bbR^n}  \left(\frac{|f(x)|}{\lambda}\right)^{p(x)}\,dx < \infty. \]
This set becomes a Banach function space when equipped with the
norm
\[ \|f\|_{p(\cdot),\bbR^n} = \inf\bigg\{ \lambda > 0 :
\int_{\bbR^n} \left(\frac{|f(x)|}{\lambda}\right)^{p(x)}\,dx \leq 1
\bigg\}.  \]
These spaces are referred to as variable $L^p$ spaces and they generalize the
standard $L^p$ spaces.  They have many properties in
common with the standard $L^p$ spaces.

We define $\P(\bbR^n)$ to be the set of measurable functions $p :
\bbR^n \rightarrow [1,\infty)$ such that
$$
p_-
=
\essinf\{ p(x) : x\in \bbR^n \} > 1, \qquad p_+
=
\esssup\{ p(x) : x\in \bbR^n \} < \infty.
$$
Under these conditions $L^{p(\cdot)}(\bbR^n)$ becomes a uniformly convex, reflexive space whose dual space \,$(L^{p(\cdot)}(\bbR^n))^*$ is equal to $L^{p'(\cdot)}(\bbR^n)$, where $p'(\cdot)$
is the conjugate exponent function defined by
\[ \frac{1}{p(x)}+\frac{1}{p'(x)} = 1, \qquad x\in\bbR^n. \]

Let  $\B(\bbR^n)$ be the set of $p(\cdot)\in\P(\bbR^n)$ such that  the maximal function
$M$ is bounded on $L^{p(\cdot)}(\bbR^n)$.  Some examples of these functions are those satisfying  the following log-H\"older continuous property,
\begin{equation} \label{diening-cond}
|p(x)-p(y)| \leq \frac{C}{|\log|x-y||},  \qquad |x-y| \leq 1/2,
\end{equation}
\begin{equation}\label{infty-cond}
|p(x) - p(y)| \leq \frac{C}{\log(e+|x|)},  \qquad |y| \geq |x|.
\end{equation}
See \cite{CF} and  \cite{DHHR}  for more information about these spaces.

Because our proofs rely on duality arguments, we will not need
that the maximal operator is bounded on $L^{p(\cdot)}(\bbR^n)$ but
on its associate space $L^{p'(\cdot)}(\bbR^n)$.

Since
$$
|p'(x)-p'(y)| \leq \frac{|p(x)-p(y)|}{(p_- -1)^2},
$$
it follows at once that if $p(\cdot)$ satisfies
\eqref{diening-cond} and \eqref{infty-cond}, then so does
$p'(\cdot)$, i.e., if these two conditions hold, then $M$ is
bounded on $L^{p(\cdot)}(\bbR^n)$ and $L^{p'(\cdot)}(\bbR^n)$.

The following extrapolation theorem from \cite{CFMP} is the key estimate.

\begin{Theorem} \label{extrapol-thm}
Let $\F$ be a family of pairs of measurable functions
$(f,g)$ and suppose that for every weight $w\in A_1$,
\begin{equation*}
\int_{\bbR^n} f(x)\,w(x)\,dx
\leq
C_0\, \int_{\bbR^n} g(x)\,w(x)\,dx, \qquad (f,g)\in\F,
\end{equation*}
where $C_0$ depends only on the $A_1$ constant of $w$.
Let $p(\cdot)\in \P(\bbR^n)$ be such that $p'(\cdot)\in\B(\bbR^n)$.  Then for all $(f,g)\in\F$ such
that $f\in L^{p(\cdot)}(\Omega)$,
\begin{equation*}
\|f\|_{p(\cdot),\Omega} \leq C\,\|g\|_{p(\cdot),\Omega},
\end{equation*}
where the constant $C$ is independent of the pair $(f,g)$.
\end{Theorem}

\begin{Theorem} Let $p(\cdot) \in \P(\bbR^n)$ be such that $p'(\cdot)\in\B(\bbR^n)$.   Then:
\begin{itemize}
\item[a) ] {\bf Scalar context}.
\begin{equation*}
\|Tf\|_{L^{p(\cdot)}(\bbR^n)} \le c\, \|Mf\|_{L^{p(\cdot)}(\bbR^n)},
\end{equation*}
for any smooth function such that the left-hand side is finite.

\item[b) ] {\bf Vector-valued extension}.  If $q\in (0,\infty)$ then
\begin{equation*}
\Big\|
\Big(\sum_j |T f_j|^q\Big)^{1/q} \Big\|_{L^{p(\cdot)}}
\le
c\, \Big\| \Big(\sum_j (Mf_j)^q\Big)^{1/q} \Big\|_{L^{p(\cdot)}}
\end{equation*}
and if $q>1$  this is bounded by
$c\,\| M(\|f\|_{\ell^q}) \|_{L^{p(\cdot)}}$.
%
\end{itemize}
\end{Theorem}
The proof follows the same scheme as the proof of Theorem \ref{BFS} from the previous section using the main ideas from \cite{CFMP}.

Therefore, we get the following.

\begin{Corollary}Let $p(\cdot) \in \P(\bbR^n)$ be such that both $p(\cdot)$ and  $p'(\cdot)$ belong to $\B(\bbR^n)$. Then if \,$X=L^{p(\cdot)}(\bbR^n)$ we have that $T: X\to X$ and also $T: X_{\ell^{q}} \to X_{\ell^{q}}.$
\end{Corollary}

\section{} \label{B}
In this appendix, we provide a different proof of our Theorem \ref{Thm:omegalq}. It is actually given in \cite{B}, but without control on the constant.  Take $\bar B(t)= t^{\frac 12(\frac pr+1)}$, it is easy to check $\bar B(t)\in B_{p/r}$. Observe that for any weight $w$ and Young function $A$ such that $\bar A \in B_{p'}$, we have
\begin{align*}
\sup_Q \|w^{1/p}\|_{A, Q}\|(M_{A_p}w)^{-r/p}\|_{B, Q}^{1/r}&\le \sup_Q\inf_{x\in Q} (M_{A_p}w)^{\frac 1p}\|(M_{A_p}w)^{-r/p}\|_{B, Q}^{1/r}\le 1.
\end{align*} Recall that $v=M_{A_p} w$. Now we have,
\begin{align*}
\|\mathcal A_{r,\mathcal S}(f)\|_{L^p(w)}&= \sup_{\|g\|_{L^{p'}}=1}\int \mathcal A_{r,\mathcal S}(f) w^\frac1pg\\
&= \sup_{\|g\|_{L^{p'}}=1} \sum_{Q\in \mathcal S}\langle f^r v^{\frac rp} v^{-\frac rp}\rangle_Q^{\frac 1r}\int_Q w^{\frac 1p} g\\
&\le 4 \sup_{\|g\|_{L^{p'}}=1} \sum_{Q\in \mathcal S}\| f^r v^{\frac rp} \|_{\bar B, Q}^{\frac 1r} \|v^{-\frac rp}\|_{B, Q}^{\frac 1r} \|w^{\frac 1p}\|_{A, Q} \|g\|_{\bar A, Q}|Q|\\
&\le 8 \sup_{\|g\|_{L^{p'}}=1} \sum_{Q\in \mathcal S}\| f^r v^{\frac rp} \|_{\bar B, Q}^{\frac 1r}  \|g\|_{\bar A, Q}|E_Q|\\
&\le 8 \sup_{\|g\|_{L^{p'}}=1}\int M_{\bar B}(f^r v^{\frac rp})^{\frac 1r} M_{\bar A}(g) \\
&\le c_n \|M_{\bar A}\|_{L^{p'}} (\beta_{p/r}(\bar B))^{\frac 1r} \|f\|_{L^p(v)},
\end{align*}where in the last step, we have used the H\"older's inequality and Lemma \ref{lem:Young}. A direct calculation yields
\[
\beta_{p/r}(\bar B)= \int_1^\infty \frac{t^{\frac 12(\frac pr+1)}}{t^{p/r}}\frac {dt}{t}=\frac {2r}{p -r}.
\]
Altogether, we obtain
\begin{equation}\label{eq:al}
\|\mathcal A_{r,\mathcal S}(f)\|_{L^p(w)}\le c_n \Big(\frac {2r}{p -r}\Big)^{\frac 1r}\|M_{\bar A}\|_{L^{p'}} \|f\|_{L^p(v)}.
\end{equation}
For the Calder\'on-Zygmund case, namely, $r=1$, \eqref{eq:al} turns to
\[
\|\mathcal A_{r,\mathcal S}(f)\|_{L^p(w)}\le c_n p'\|M_{\bar A}\|_{L^{p'}} \|f\|_{L^p(v)}.
\]
For $T_\Omega$ with $\Omega\in L^\infty$, choosing $r= 1+\frac{p-1}2$, we obtain
\[
\|T_{\Omega}(f)\|_{L^p(w)}\le c_n(p')^2\|M_{\bar A}\|_{L^{p'}} \|f\|_{L^p(v)},
\]which coincides with our previous result.

\section*{Acknowledgements}

We thank Francesco Di Plinio for sending us a preprint of \cite{CACDPO} and for some useful comments calling our attention to some results from the paper by David Beltran \cite{B}.  We would also like to thank David Beltran
for telling us his qualitative proof of Theorem \ref{Thm:omegalq}.

\end{document}